\documentclass[10pt,a4paper]{article}

\usepackage{float}
\usepackage{graphicx}
\usepackage{amssymb,amsmath,amsthm}
\usepackage{esint}
\usepackage{tikz}
\usepackage{bbm,dsfont,bm}
\usepackage{color}
\usepackage{enumerate,fancyhdr}
\usepackage[english]{babel}
\usepackage[hypertexnames=false]{hyperref}
\usepackage{thmtools}
\usepackage[left=2cm,right=2cm,top=2cm,bottom=2cm]{geometry}

\newtheorem{theorem}{Theorem}[section]
\newtheorem{lemma}[theorem]{Lemma}

\newtheorem{definition}[theorem]{Definition}

\newtheorem{remark}[theorem]{Remark}
\newtheorem{proposition}[theorem]{Proposition}
\newtheorem*{theorem*}{Theorem}

\newcommand{\norm}[1]{{\left\Vert{#1}\right\Vert}}

\numberwithin{equation}{section}

\title{Time-periodic solutions for viscous fluids interacting with nonlinear Koiter plates}
\author{Claudiu M\^{i}ndril\u{a}\thanks{``Gheorghe Mihoc--Caius Iacob'' Institute of Mathematical Statistics and Applied Mathematics of the Romanian Academy, Calea 13 Septembrie 13, Bucharest, Romania. \texttt{mndclaudiu@gmail.com}}}
\date{\today}

\begin{document}

\maketitle

\begin{abstract}
We prove the existence of time-periodic weak solutions for a fluid-structure interaction
system coupling the incompressible Navier-Stokes equations in a three-dimensional moving
domain with a nonlinear Koiter plate equation on its upper boundary. The lateral boundary
is space-periodic ($\omega = \mathbb{T}^{2}$), a natural setting for flow in pipes and
channels of periodic cross-section driven by a time-periodic pressure gradient, and the
fluid satisfies a no-slip coupling condition at the moving interface. The elastic energy
of the plate is governed by the nonlinear Koiter model, which yields an $H^2$-coercive
operator and accounts for both membrane and bending effects.

To the best of our knowledge, this is the first result on time-periodic weak solutions
for a fluid-structure interaction system with a \emph{nonlinear} elastic energy.
The main novelty, compared to our earlier works~\cite{Claudiu22, mindrila-schwarz-shell}
on the linear case — a linear elastic plate and a linear Koiter shell respectively —
is the replacement of a two-stage fixed-point procedure — a Leray-Schauder argument
at the discrete level followed by a set-valued Kakutani-Glicksberg-Fan argument at
the continuous level — by a \emph{single} Leray-Schauder fixed point applied directly
to the fully coupled Galerkin system. This reduction is not merely a simplification:
the nonlinearity of the Koiter energy destroys the convexity of the solution map on
which Kakutani-Fan relies, making the two-stage approach of~\cite{Claudiu22}
unavailable and the single fixed point the only viable strategy.
The single fixed point is made consistent by the operator $\mathcal{P}_\varepsilon$,
which artificially periodises the prescribed geometry so that the energy comparison
$E_n(0) \leq E_n(T)$ is well-posed at each step of the approximation.
A further key ingredient is the coercivity identity
$\langle K'(\eta), \eta \rangle \ge 2K(\eta)$,
which underlies the uniform-in-time energy estimate and holds for flat reference
configurations but fails for general Koiter shells, explaining why the present paper
treats plates rather than shells.
\end{abstract}



\section{Introduction}\label{sec:intro}

\subsection{Background and motivation}

The mathematical analysis of fluid-structure interaction (FSI) systems has attracted significant
attention over the past two decades, driven both by the richness of the underlying mathematics
and by physical applications ranging from cardiovascular haemodynamics to pipeline engineering.
A prototypical model couples the incompressible Navier-Stokes equations in a time-dependent
domain with an elastic structure that forms part of the boundary, the two subsystems being
linked by kinematic and dynamic coupling conditions at the moving interface.

For the Cauchy problem, the existence of weak solutions is by now well understood in a number
of configurations. In the case of a viscous fluid coupled with an elastic plate governed by a
linear operator, weak solutions were first constructed by Chambolle, Desjardins, Esteban and
Grandmont~\cite{chambolle2005existence} and subsequently extended and refined in
\cite{Gr08, LR14, Lengeler16}. The nonlinear case, where the elastic energy of the structure is
governed by the nonlinear Koiter model, was treated in the fundamental work of
Muha and Schwarzacher~\cite{MS22}, who established existence and regularity of weak solutions
for a viscous fluid interacting with a nonlinear shell in three dimensions, extending the
earlier contribution of Muha and \v{C}ani\'{c}~\cite{MC15} which dealt with nonlinear Koiter
membrane energy. We also mention the work of Lengeler and R\r{u}\v{z}i\v{c}ka~\cite{LR14}
on generalized Newtonian fluids with Koiter-type shells, and the variational approach of
Bene\v{s}ov\'{a}, Kampschulte and Schwarzacher~\cite{BKS23}.

The \emph{nonlinear Koiter model} used throughout this paper is one of the most widely adopted
models in the mathematical theory of thin elastic shells. It was introduced and rigorously
justified by Ciarlet and collaborators~\cite{ciarlet2000theory, Ciarlet05, ciarlet2018nonlinear},
and it accounts for both membrane and bending effects through the change-of-metric tensor
$\mathbb{G}(\eta)$ and the change-of-curvature tensor $\mathbb{R}^\sharp(\eta)$.
In the flat geometry considered here, these reduce to
$\mathbb{G}_{ij}(\eta) = \partial_i\eta\,\partial_j\eta$ and
$\mathbb{R}^\sharp_{ij}(\eta) = \partial_{ij}\eta$,
and the Koiter energy takes the simplified form
\[
K(\eta) \simeq \int_\omega |\nabla\eta|^4 + |\nabla^2\eta|^2\,dA,
\]
which is coercive in $H^2(\omega)$; see also~\cite{ciarlet-roquefort-2001} for the
coercivity of the elasticity tensor $\mathcal{A}$ in this context.
The $H^2$-coercivity is essential throughout our analysis, as it provides the
regularity needed to control both the fluid and the plate uniformly in time.

The \emph{time-periodic} problem, by contrast, has received comparatively little attention
despite its physical relevance. A periodic forcing arises naturally in the modelling of
flow in elastic pipes and blood vessels, where the driving pressure is imposed by a
periodically beating heart, or in industrial pipelines subject to cyclic boundary loads.
For the Navier-Stokes equations alone, time-periodic solutions have been studied since the
pioneering work of Prouse~\cite{Pr63}; see also the results of Galdi and
collaborators~\cite{GS06, G13, G20} on flows around moving bodies. In the FSI context,
Casanova~\cite{Casanova} proved the existence of time-periodic strong solutions
for a fluid-structure system with a damped beam, exploiting maximal $L^p$-regularity.
The first existence result for time-periodic \emph{weak} solutions to a FSI system
with a nonlinearly coupled moving boundary was obtained in our earlier
work~\cite{Claudiu22}, where we treated the case of an \emph{elastic plate governed
by a linear operator}. This was subsequently extended in~\cite{mindrila-schwarz-shell}
to the case of a \emph{linear Koiter shell}, allowing for a curved reference geometry
and dynamic pressure boundary conditions on a fixed part of the boundary.
We also mention the recent work of Kreml, M\'{a}cha, Ne\v{c}asov\'{a} and
Trifunovi\'{c}~\cite{kreml} on time-periodic solutions for compressible fluids
interacting with viscoelastic beams, and the work of Mosny, Muha, Schwarzacher and
Webster~\cite{mosny2024time} on time-periodic solutions for general hyperbolic-parabolic
systems which provides an abstract framework applicable to linearised FSI.

The present paper, based on author's PhD thesis~\cite{thesis-mindrila} is, to the best of our knowledge, the \emph{first} to establish
time-periodic weak solutions for a FSI system with a \emph{nonlinear} elastic energy.
The prior works~\cite{Claudiu22, mindrila-schwarz-shell} treat the linear case —
a linear elastic plate and a linear Koiter shell respectively — and both rely on
a two-stage fixed-point procedure in which a set-valued Kakutani-Glicksberg-Fan
argument at the continuous level is used to recover the geometry coupling.
This approach is fundamentally blocked by the nonlinear Koiter energy, as explained
in Section~\ref{ssec:a-priori-section-paper}, and a genuinely new fixed-point strategy
is required.

\subsection{Setting and main result}

We consider an incompressible viscous fluid occupying the moving domain
\[
\Omega_{\eta(t)} := \left\{ (x,y,z)\in\mathbb{R}^{3} : (x,y)\in\omega,\; 0 < z < 1+\eta(t,x,y) \right\},
\]
whose upper boundary deforms in time.
The deformation is driven by a thin elastic plate, whose displacement from its
reference position is described by a scalar function $\eta(t,\cdot):\omega\to\mathbb{R}$,
where $\omega = \mathbb{T}^2$ is the two-dimensional torus. The choice $\omega = \mathbb{T}^2$
imposes space-periodic lateral boundary conditions on the fluid, a natural model
for flow in pipes and channels of periodic cross-section driven by a time-periodic
pressure gradient. The fluid satisfies the incompressible Navier-Stokes equations
in $\Omega_{\eta(t)}$, with a no-slip coupling condition at the moving interface,
and is subject to an external, time-periodic body force $\mathbf{f}$, with period $T$. The plate satisfies a
wave-type equation driven by a time-periodic load $g$ and the nonlinear Koiter
elastic energy
\[
K(\eta) \simeq \int_\omega |\nabla\eta|^4 + |\nabla^2\eta|^2\,dA,
\]
which is $H^2$-coercive. 

It seems natural that, therefore, $\mathbf{u}$ and $\eta$ should be
time-periodic with the same period $T>0$.
The aim of this work is to prove that such a solution exists.

The total energy of the system at time $t$ is
\[
E(t) := \frac{1}{2}\int_{\Omega_{\eta(t)}} |\mathbf{u}|^2\,dx
       + \frac{1}{2}\int_\omega |\partial_t\eta|^2\,dA + K(\eta(t)),
\]
the three terms being respectively the kinetic energy of the fluid, the kinetic
energy of the plate, and the elastic energy. Two further quantities enter the
main result: the combined forcing norm
\[
C(\mathbf{f},g) := \int_{I}\!\!\int_{\Omega_{\eta(t)}} |\mathbf{f}|^2\,dx\,dt
                 + \int_{I}\!\!\int_\omega |g|^2\,dA\,dt,
\]
and the conserved mean displacement $m := \int_\omega \eta\,dA$, which is
constant in time as a consequence of the incompressibility constraint and the
no-slip condition (see Remark~\ref{rmk:int-eta-constant-time}).

The main result of this paper, Theorem~\ref{thm:main}, states that if $C(\mathbf{f},g)$
and $m^2$ are bounded by a sufficiently small constant depending only on the data of
the problem, then there exists at least one time-periodic weak solution $(\mathbf{u},\eta)$
satisfying the uniform energy estimate
\[
\sup_{t \in [0,T]} E(t) \lesssim C(\mathbf{f},g)^2 + C(\mathbf{f},g) + m^2.
\]
The smallness condition is not merely a technical artifact: it ensures that the plate
displacement remains bounded away from zero ($\|\eta\|_{L^\infty} < \kappa < 1$),
preventing the fluid domain from degenerating, and it reflects the genuinely nonlinear
character of the problem in which large forcing could drive the plate to contact with
the bottom of the cavity.

\subsection{Proof strategy and new contributions}

The proof combines three ingredients, each of which requires genuinely new arguments with
respect to the linear case treated in~\cite{Claudiu22}.

\medskip
\noindent\textit{Uniform energy estimate.} For time-periodic solutions there is no
decay mechanism, and one must obtain a bound on $\sup_{t} E(t)$ that is independent
of the initial data. This is achieved by testing the elastic equation with $\eta$ itself,
which requires constructing a corresponding divergence-free test function for the fluid
equation that extends $\eta$ to the interior of $\Omega_{\eta(t)}$. The extension
operator $\mathcal{F}_\eta$ (Proposition~\ref{chap-appendix-prop-extension-plate}) provides
this, and the key estimate \eqref{eqn:formal-claim} is then a consequence of the $H^2$-coercivity
of the nonlinear Koiter operator $K'(\eta)$ and the \emph{coercivity identity}
\begin{equation}\label{eqn:Kprime-coercivity}
\langle K'(\eta), \eta \rangle \ge 2K(\eta),
\end{equation}
which follows by a direct computation from the explicit form of $K(\eta)$ in the
\emph{flat} case $\Gamma = \omega \times \{1\}$.
We stress that \eqref{eqn:Kprime-coercivity} is specific to the flat reference geometry
and does \emph{not} hold for general Koiter shells, where the Koiter energy involves
the full change-of-curvature tensor relative to a curved reference surface.
This is the precise reason why the present paper treats \emph{Koiter plates} (flat
reference configuration) rather than \emph{Koiter shells} (curved reference
configuration): the identity \eqref{eqn:Kprime-coercivity} is the cornerstone of
the uniform energy estimate, and without it the entire strategy breaks down.
The presence of the nonlinear term $|\nabla\eta|^4$ in $K(\eta)$ makes this argument
qualitatively different from the linear case, where $K(\eta) = \|\nabla^2\eta\|^2_{L^2}$
and \eqref{eqn:Kprime-coercivity} holds trivially with equality.

\medskip
\noindent\textit{A single fixed point simultaneously enforcing periodicity and coupling.}
A central novelty of the present paper, compared to our earlier work~\cite{Claudiu22},
is the reduction from \emph{two} fixed-point arguments to \emph{one}.
In~\cite{Claudiu22}, existence was established via a two-stage procedure:
a first Leray-Schauder fixed point at the discrete Galerkin level to obtain
periodicity of the approximate solution, followed by a second fixed point
— of set-valued type, via the Kakutani-Glicksberg-Fan theorem — at the continuous
level to recover the coupling $\delta = \eta$ in the limit.
Here we show that both goals can be achieved simultaneously by a \emph{single}
Leray-Schauder fixed point applied directly to the fully coupled Galerkin system.

The key mechanism is the operator $\mathcal{P}_\varepsilon$
(Lemma in Section~\ref{sec:proof}), which \emph{artificially forces the prescribed
geometry $\delta_n$ to be time-periodic} by replacing it with a linear interpolant
on $[T-\varepsilon, T]$ that joins $\delta_n(T-\varepsilon)$ to $\delta_n(0)$.
This periodisation is essential: without it, the comparison of $E_n(0)$ and
$E_n(T)$ — which requires integrating over the same domain at both endpoints
— would be ill-defined when $\Omega_{\delta_n(0)} \ne \Omega_{\delta_n(T)}$.
The mapping $\mathcal{T}\colon \mathbf{a}_n \mapsto \mathbf{b}_n$ then assigns
to each periodised prescribed geometry $\delta_n$ the solution $(\mathbf{u}_n,\eta_n)$
of the decoupled system with initial values $\mathbf{b}_n(0) = \mathbf{a}_n(T)$,
$\mathbf{b}_n'(0) = \mathbf{a}_n'(T)$. A fixed point $\mathbf{a}_n = \mathbf{b}_n$
simultaneously gives $\eta_n = \delta_n$ (coupling) and $\eta_n(0) = \eta_n(T)$
(periodicity), eliminating entirely the need for a second fixed point.

This single fixed-point strategy is not merely a simplification of the approach
in~\cite{Claudiu22}: it is \emph{forced} by the nonlinear Koiter energy.
In~\cite{Claudiu22}, the second fixed point — of set-valued type via the
Kakutani-Glicksberg-Fan theorem — relied crucially on the \emph{convexity} of
the solution map $\delta \mapsto (\mathbf{u}, \eta)$ at the continuous level.
This convexity in turn depended on the \emph{linearity} of the plate equation:
with a linear elastic operator, the solution set corresponding to a given geometry
$\delta$ is convex, which is exactly what is needed to apply Kakutani-Fan.
For the nonlinear Koiter energy, the plate equation is nonlinear and the solution
map is no longer convex; Kakutani-Fan is therefore unavailable, and the two-stage
approach of~\cite{Claudiu22} is fundamentally blocked. The single Leray-Schauder
fixed point at the Galerkin level — made consistent by the artificial periodisation
$\mathcal{P}_\varepsilon$ — is, as far as we are aware, the only viable strategy
to simultaneously recover periodicity and coupling in this setting.

\medskip
\noindent\textit{Refined $L^2$ compactness argument.} Passing to the limit $n\to\infty$
in the Galerkin scheme requires strong convergence of $(\mathbf{u}_n, \partial_t\eta_n)$
in $L^2_t L^2_x$. This is the most technically demanding part of the proof and goes
well beyond what is needed for the Cauchy problem. The key tool is the divergence-free
extension operator $\mathcal{F}_\eta$, which lifts a scalar function $\xi$ on $\omega$
to a divergence-free vector field on $\Omega_{\eta(t)}$ with $\text{tr}_\eta(\mathcal{F}_\eta\xi)
= \xi\mathbf{e}_3$ (Proposition~\ref{chap-appendix-prop-extension-plate}). This operator
allows one to decompose the $L^2$ norm of the fluid velocity as
\[
\int_{\Omega_{\eta_n(t)}} |\mathbf{u}_n|^2\,dx
= \int_{\Omega_{\eta_n(t)}} \mathbf{u}_n \cdot \mathcal{F}_{\eta_n}(\partial_t\eta_n)\,dx
+ \int_{\Omega_{\eta_n(t)}} \mathbf{u}_n \cdot
  \bigl(\mathbf{u}_n - \mathcal{F}_{\eta_n}(\partial_t\eta_n)\bigr)\,dx.
\]
The first term captures the component of $\mathbf{u}_n$ aligned with the plate velocity,
and its convergence follows from a mollification argument and the equicontinuity in time
of a family of linear functionals acting on the Galerkin solutions.
Convergence of the second term — the purely solenoidal remainder — is substantially
harder and requires the uniform convergence estimate~\eqref{eqn:uniform-conv-h1-piola},
which is established via the Galerkin projection $\mathcal{P}_n$ composed with the
Piola map $\mathcal{J}_{\delta_{\tilde\sigma}}$, together with a H\"{o}lder
equicontinuity estimate that uses $\mathbf{u}_n \in L^{8/3}_t L^4_x$ and the explicit
computation of $\partial_t \mathcal{J}_{\eta_n}$. The use of the Galerkin projection
$\mathcal{P}_n$ — rather than a direct approximation in the limit space — is essential
to maintain compatibility with the finite-dimensional system during the limit passage
and constitutes a further novelty of our approach; see also~\cite{MS22, LR14} for
related compactness arguments in the Cauchy and linear time-periodic settings.

\subsection{Open problems}

The extension of the present results to \emph{general Koiter shells}
(i.e.\ non-flat reference configurations) remains open. For curved reference surfaces,
the Koiter energy involves the full change-of-curvature tensor, and the identity
$\langle K'(\eta),\eta\rangle \ge 2K(\eta)$, which is the cornerstone of our uniform
energy estimate (see \eqref{eqn:Kprime-coercivity} and the discussion in
Section~\ref{ssec:a-priori-section-paper}), \emph{fails} in general.
Establishing an analogue of this identity for curved shells, or finding an alternative
route to the uniform bound on $\sup_t E(t)$, appears to be the key obstacle;
see \cite{KamSchSpe23} for recent progress on nonlinear shells in the Cauchy problem setting.

One possible avenue to circumvent this difficulty is to replace the Koiter model with
an alternative shell theory for which better coercivity properties hold on curved
reference configurations. A natural candidate is the \emph{Cosserat (micropolar) shell
model}, in which the shell kinematics are enriched by an independent rotation field.
This additional structure leads to energy functionals with improved coercivity and
has been studied rigorously by Ghiba, B\^{i}rsan, Neff and
collaborators~\cite{GhibaBirsanNeff2023, GhibaNeff2023, MohammadiSaemGhibaNeff2023};
whether such models are compatible with the FSI framework and permit a uniform
energy estimate of the type \eqref{eqn:Kprime-coercivity} is an open question. See also~\cite{ghiba2026nonlinearkirchhoffloveshellmodels}.

\section{Setting}\label{sec:setting}
\subsection{Geometry}

Let $\Omega \subset \mathbb{R}^{3}$ be a cube of radius $1$, denoted by $\Omega=\left[0,1\right]^{3}$. On the lateral boundaries we will impose space-periodic conditions and thus write

\begin{equation*}
\Omega:=\left\{ \left(x,y,z\right):\ \left(x,y\right)\in\mathbb{T}^{2},\ z\in\left(0,1\right)\right\} . 
\end{equation*}
with $\mathbb{T}^{2}$ the 2 dimensional torus.

The upper part of the cube, at $z=1$ is assumed to be  the \emph{flexible part of the boundary} and we denote it by
\begin{equation}
\Gamma:=\left\{ \left(x,y,1\right):\left(x,y\right)\in\mathbb{T}^{2}\right\} .
\end{equation}
Denoting $\omega:=\mathbb{T}^{2}=\mathbb{R}^{2}/\mathbb{Z}^{2}$ the unit torus we define a parametrization of $\Gamma$ via the mapping \[
\phi:\omega\mapsto\Gamma\subset\mathbb{R}^{3},\quad\phi\left(x,y\right)=\left(x,y,1\right),\ \left(x,y\right)\in\omega.\] 

We assume further that at $\Gamma$ a thin elastic plate is attached which,  
at each time $t\in I$  evolves in the  normal direction 
\[
\mathbf{n}=\mathbf{e}_{3}=(0,0,1).
\]
Let us denote
the displacement of the elastic plate by $\boldsymbol{\eta}:\omega \mapsto \mathbb{R}^{3}$.

We can write $\boldsymbol{\eta}=\eta \mathbf{e}_{3}$ and we will study  the scalar unknown
\[
\eta: I\times \omega \mapsto \mathbb{R},\]
where $I := [0,T]$ denotes the time interval for a fixed period $T > 0$.

At each $t\in I$ we obtain a new \emph{deformed boundary}  denoted 
\begin{equation}
\Gamma_{\eta}\left(t\right):=\left\{ \left(x,y,1+\eta\left(t,x,y\right)\right):\left(x,y\right)\in\omega\right\} .
\end{equation}
which can be parametrized by 
\begin{equation*}
\phi_{\eta\left(t\right)}:\omega\mapsto\mathbb{R}^{3},\ \phi_{\eta\left(t\right)}\left(x,y\right):=\left(x,y,1+\eta\left(t,x,y\right)\right).
\end{equation*}

We may now define the  tangent vectors $\boldsymbol{\tau}_{1}^{\eta},\boldsymbol{\tau}_{2}^{\eta}$ at any point $p=\phi_{\eta\left(t\right)}\left(\theta ,z\right)\in \Gamma ^{\eta}(t)$ with $(x,y)\in \omega$ by 
\begin{equation*}
\boldsymbol{\tau}_{i}^{\eta}\left(p\right):=\left.\partial_{i}\phi_{\eta\left(t\right)}\right|_{\left(x,y\right)},\quad i=1,2
\end{equation*}
and the  outer normal vector $\mathbf{n}^{\eta}$ at $p$ via 
\begin{equation*}
\mathbf{n}^{\eta}\left(p\right):=\left.\partial_{1}\phi_{\eta\left(t\right)}\times\partial_{2}\phi_{\eta\left(t\right)}\right|_{\left(x,y\right)}.
\end{equation*}

Thus, at each $t\in I$ we obtain the moving domains $\Omega_{\eta}\left(t\right)$ given
by
\begin{equation*}
\Omega_{\eta}\left(t\right):=\left\{ \left(x,y,z\right)\in\mathbb{R}^{3}:\left(x,y\right)\in\omega,0<z<1+\eta\left(t,\theta,z\right)\right\} .
\end{equation*}
 The setting is sketched in Figure~\ref{fig:moving-domains}.

To ensure that the domains $\Omega_{\eta}\left(t\right)$ do not degenerate 
we  will assume  throughout this work that 
\begin{equation}\label{eqn:condition-eta-norm}
\left\Vert \eta\right\Vert _{L_{t,x}^{\infty}}\le \kappa<1.
\end{equation}
The condition \eqref{eqn:condition-eta-norm} can be ensured by controlling the whole energy of the system, see \eqref{eqn:energy-def} and Theorem~\ref{thm:main} and also Remark~\ref{rmk:bound-eta-norm} for further details.

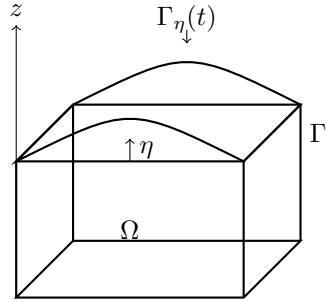
\begin{figure}[H]
\centering
\begin{tikzpicture}[scale=3]

% z-axis only (clean)
\draw[->] (0,0) -- (0,1.2) node[above] {$z$};

% Parameters
\def\h{0.6}
\def\s{0.25}

% Bottom face
\draw[thick] (0,0) -- (1,0) -- (1+\s,\s) -- (\s,\s) -- cycle;

% Vertical edges
\draw[thick] (0,0) -- (0,\h);
\draw[thick] (1,0) -- (1,\h);
\draw[thick] (1+\s,\s) -- (1+\s,\h+\s);
\draw[thick] (\s,\s) -- (\s,\h+\s);

% Reference top Gamma
\draw[thick] (0,\h) -- (1,\h) -- (1+\s,\h+\s) -- (\s,\h+\s) -- cycle;
\node[right] at (1+\s,\h+0.5*\s) {$\Gamma$};

% Deformed top with FIXED endpoints
\draw[thick]
(0,\h) 
.. controls (0.5,\h+0.25) .. (1,\h)
-- (1+\s,\h+\s)
.. controls (0.5+\s,\h+\s+0.25) .. (\s,\h+\s)
-- cycle;

% Gamma^eta label: above the back hump with a small leader arrow
\node at (0.5+\s,\h+\s+0.38) {$\Gamma_{\eta}(t)$};
\draw[->] (0.5+\s,\h+\s+0.33) -- (0.5+\s,\h+\s+0.27);

% Displacement arrow: centred, short, stays within the gap below the deformed surface
\draw[->] (0.5,\h) -- (0.5,\h+0.10);
\node[right] at (0.5,\h+0.05) {$\eta$};

% Domain label
\node at (0.5,0.3) {$\Omega$};

\end{tikzpicture}
    \caption{The moving domains}
    \label{fig:moving-domains}
\end{figure}

% \begin{remark}
%     The assumption that the displacement of the shell is restricted to the normal  direction can be relaxed to include displacements in all three directions, as long as the injectivity of $\phi_{\eta}$ can be guaranteed; see  \cite[p. 6639]{galic-canic-muha} for further details. 
% \end{remark}
\subsection{The equations}\label{ssec:problem}
Let us denote, here and throughout the work, the time-space domains by  
\[
I\times\Omega_{\eta}:=\bigcup_{t\in I}\left\{ t\right\} \times\Omega_{\eta}\left(t\right).
\]

\paragraph{The fluid equations.}
We assume that at each $t\in I$ the domain $\Omega_{\eta}\left(t\right)$ is filled with  fluid, described below.
Let $\mathbf{u}:I\times\Omega_{\eta}\mapsto\mathbb{R}^{3}$ denote the velocity field of the fluid and $p:I\times \Omega_{\eta}\mapsto \mathbb{R}$ is the associated pressure field associated to $\mathbf{u}$. We assume the fluid to be homogeneous, of density $\rho_{f}>0$, incompressible, viscous with viscosity $\mu_{f}>0$ and obeying the \emph{Navier-Stokes equations}. This means that we have 

\begin{equation}\label{eqn:fluid}
 \left.\begin{array}{c}
\rho_{f}\left(\partial_{t}\mathbf{u}+\mathbf{u}\cdot\nabla\mathbf{u}\right)=\text{div}\ \sigma+\mathbf{f}\\
\text{div}\ \mathbf{u}=0
\end{array}\right\} \text{in}\ I\times\Omega_{\eta}
\end{equation}
with $\sigma$ denoting the usual \emph{Cauchy stress tensor} given by the formula 
\begin{equation}
\sigma=\sigma(x,\mathbf{u},p):=2\mu_{f}\frac{\nabla\mathbf{u}+\left(\nabla\mathbf{u}\right)^{T}}{2}-p\mathbb{I}_{3\times3}=:2\mu_{f}\mathbb{D}\mathbf{u}-p\mathbb{I}
\end{equation}
and $\mathbb{D}$ represents the symmetric gradient. 
By $\mathbf{f}:I\times\mathbb{R}^{3}\mapsto\mathbb{R}^{3}$ we denote a time-periodic external force, which we assume to be given and to belong in $L^{2}\left(I;\mathbb{R}^{3}\right)$ for simplicity. 

We set $\rho_{f}=2\mu_{f}=1$.

We will neglect the pressure, by dealing with test-functions with divergence zero. 
Once found a velocity field $\mathbf{u}$, the pressure can then be reconstructed by standard methods.

\paragraph{The elasticity equations.}

The elastic plate is assumed to be a homogeneous medium of given density $\rho_{s}>0$ and of thickness $h>0$. 
We denote its displacement, measured with respect to the (fixed) lateral boundary $\Gamma$, by $\eta:I\times\omega \mapsto \mathbb{R}$ and we assume that it fulfills a Lam\'{e} type equation of the form 
\begin{equation}\label{eqn:lame}
\rho_{s}h\partial_{tt}\eta+K^{\prime}\left(\eta\right)=G+g
\end{equation}
with $g\in L^{2}\left(I;L^{2}\left(\omega\right)\right)$ accounting for possible gravitational forces, assumed to be time-periodic as well. 

The force $G$ is described in \eqref{eqn:G-force}.
We set $\rho_{s}=h=1$. The equation~\eqref{eqn:lame} can be seen as the Euler-Lagrange equation associated to the following energy:
\begin{equation}\label{eqn:energy-linear-shell}
E_{\eta}:=\frac{1}{2}\int_{\omega}\left|\partial_{t}\eta\right|^{2}\ dA+\frac{1}{2}K\left(\eta\right)-\int_{\omega}g\eta\ dA.
\end{equation}
The term $K(\eta)$ denotes the \emph{nonlinear elastic Koiter energy} which is detailed below, represents a popular and frequently used model in the literature on elasticity theory, see e.g. \cite{ciarlet2018nonlinear}, \cite[Part B]{ciarlet2000theory} and the references therein.

Let us define:
\begin{itemize}
    \item The \emph{change of metric tensor} $\mathbb{G}\left(\eta\right)\in \mathbb{R}^{2\times2}$ via 

\begin{equation}
\mathbb{G}_{ij}\left(\eta\right):=\partial_{i}\phi\left(\eta\right)\partial_{j}\phi\left(\eta\right)-\partial_{i}\phi\partial_{j}\phi,\quad i,j=1,2 
\end{equation}

\item The \emph{change of curvature tensor} $\mathbb{R}^{\sharp}\left(\eta\right)\in\mathbb{R}^{2\times2}$ via

\begin{equation}
\mathbb{R}_{ij}^{\sharp}:=\frac{\partial_{ij}\phi_{\eta}\cdot\mathbf{n}^{\eta}}{\left|\partial_{i}\phi\times\partial_{j}\phi\right|}-\partial_{ij}\phi\cdot\mathbf{n},\quad i,j=1,2
\end{equation}

\item The \emph{nonlinear Koiter elastic energy} as
\begin{equation}\label{eqn:nonl-koiter}
K\left(\eta\right):=K\left(\eta, \eta \right):= \frac{h}{6}\int_{\omega}\mathcal{A}\mathbb{G}\left(\eta\right):\mathbb{G}\left(\eta\right)dA+\frac{h^{3}}{48}\int_{\omega}\mathcal{A}\mathbb{R}^{\sharp}\left(\eta\right):\mathbb{R}^{\sharp}\left(\eta\right)dA
\end{equation}
\end{itemize}
By $\mathcal{A}$ we denote a fourth-order tensor  given by

\begin{equation}
\mathcal{A}\mathbf{E}:=\frac{4\lambda\mu}{\lambda+\mu}\left(\mathbf{A}:\mathbf{E}\right)\mathbf{A}+4\mu\mathbf{A}\mathbf{E}\mathbf{A},\quad\mathbf{E}\in\text{Sym}\left(\mathbb{R}^{2\times2}\right).
\end{equation}

The coefficients $\lambda, \mu$ represent the Lam\'{e} coefficients of the shell, while $\mathbf{A}$ denoting the contravariant metric tensor associated to $\Gamma$.
See \cite[p.162]{ciarlet2000theory} for further details  and \cite{ciarlet-roquefort-2001} for further details about $\mathcal{A}$ and the model, where in particular it is shown that that $\mathcal{A}$ is coercive. 

In our case (with $\Gamma$ being flat) we can simplify the computations to obtain that

 \begin{equation}
        \mathbb{G}\left(\eta\right):=\left(\partial_{i}\eta\partial_{j}\eta\right)_{i,j=1,2}\quad\mathbb{R}^{\sharp}\left(\eta\right)=\left(\partial_{ij}\eta\right)_{i,j=1,2}
    \end{equation}
    
 We see that without loss of generality we may assume
 \begin{equation}
     K\left(\eta\right)=\int_{\omega}\left|\nabla\eta\right|^{4}+\left|\nabla^{2}\eta\right|^{2}dA
 \end{equation}
 with

\begin{equation}
\begin{aligned}K\left(\eta\right)\gtrsim & \left\Vert \mathbb{G}\left(\eta\right)\right\Vert _{L_{x}^{2}}^{2}+\left\Vert \mathbb{R}^{\sharp}\left(\eta\right)\right\Vert _{L_{x}^{2}}^{2}\\
\gtrsim & \left\Vert \nabla^{2}\eta\right\Vert _{L_{x}^{2}}^{2}+\left\Vert \nabla\eta\right\Vert _{L_{x}^{4}}^{4}
\end{aligned}
\end{equation}
and  thus $K(\eta)$ is $H^{2}$ coercive.
The $H^{2}$ estimates will be our main focus since in $\mathbb{R}^{2}$  the first-order derivatives term can be controlled by the second order ones due to the embeddings \[H^{2}\left(\omega\right)\hookrightarrow\hookrightarrow W^{1,4}\left(\omega\right)\hookrightarrow\hookrightarrow L^{\infty}\left(\omega\right).
\]

Let $K^{\prime}$ denote the Fr\'{e}chet derivative, which appears in \eqref{eqn:lame}.
We obtain for an arbitrary $\xi \in H^{2}\left(\omega\right)$ that
\[
\left\langle K^{\prime}\left(\eta\right),\xi\right\rangle :=\frac{h}{6}\int_{\omega}\mathcal{A}\mathbb{G}\left(\eta\right):\left\langle \mathbb{G}^{\prime}\left(\eta\right),\xi\right\rangle dA+\frac{h^{3}}{48}\int_{\omega}\mathcal{A}\mathbb{R}^{\sharp}\left(\eta\right):\left\langle \left(\mathbb{R}^{\sharp}\left(\eta\right)\right)^{\prime},\xi\right\rangle dA
\]
with 
\[
\left\langle \left(\mathbb{G}\left(\eta\right)\right)^{\prime},\xi\right\rangle =\left(\partial_{i}\eta\partial_{j}\xi+\partial_{i}\xi\partial_{j}\eta\right)_{1\le i,j\le2},\quad\left\langle \left(\mathbb{R}^{\sharp}\left(\eta\right)\right)^{\prime},\xi\right\rangle =\left(\partial_{ij}\xi\right)_{1\le i,j\le2}.
\]

\paragraph{Boundary conditions.}

Regarding the fluid, on the (steady) lateral parts of the boundary we impose space-periodic boundary conditions.

This means $\mathbf{u}\left(t,x+1,y,z\right)=\mathbf{u}\left(t,x+1,y,z\right)$ and $\mathbf{u}\left(t,x,y+1,z\right)=\mathbf{u}\left(t,x,y+1,z\right)$ for $(x,y)\in \omega$.

On the bottom of the cavity, at $\Gamma_{b}:=\omega\times\left\{ z=0\right\}$ we assume that $\mathbf{u}=\mathbf{0}$.

Between fluid and elastic structure we assume the \emph{no-slip} boundary condition, meaning that 
\begin{equation}
\mathbf{u}\left(t,x,y,\eta\left(t,x,y\right)\right)=\partial_{t}\eta\left(t,x,y\right)\mathbf{e}_{3},\quad t\in I,\ \left(x,y\right)\in\omega 
\end{equation}

Regarding the forcing $G$, we assume the so-called \emph{dynamic coupling} boundary condition meaning that $G$ is given by the tension exerted by the fluid through $\sigma$ and evaluated in the direction $-\mathbf{e}_{3}$. This reads as 
\begin{equation}\label{eqn:G-force}
G=J_{\eta\left(t\right)}\left(x\right)\left.\left(\sigma\left(\mathbf{u},p\right)\boldsymbol{\nu}^{\eta}\right)\left(t,y\right)\right|_{y=\phi_{\eta\left(t\right)}\left(x\right)}\cdot\left(-\mathbf{e}_{3}\right),\quad x\in\omega 
\end{equation}
where we denote $\boldsymbol{\nu}^{\eta}:=\frac{\mathbf{n}^{\eta}}{\left|\mathbf{n}^{\eta}\right|}=\frac{\partial_{1}\phi_{\eta}\times\partial_{2}\phi_{\eta}}{\left|\partial_{1}\phi_{\eta}\times\partial_{2}\phi_{\eta}\right|}$ the unit outer normal vector field of $\Gamma_{\eta}$, 
$J_{\eta\left(t\right)}:=\left|\partial_{\theta}\phi_{\eta\left(t\right)}\times\partial_{z}\phi_{\eta\left(t\right)}\right|$ the Jacobian of the transformation $\phi_{\eta}$.
By simple computations we obtain  
\begin{equation}\label{eqn:jacobian}
J_{\eta}=\sqrt{\left|\nabla\eta\right|^{2}+1}.
\end{equation}
 
% The problem of evolution after a possible contact (although clearly interesting and widely open) is beyond the aims of this paper, therefore we will focus on the following condition 
%  \begin{equation}
%      \left\Vert \eta\right\Vert _{L^{\infty}\left(I\times\omega\right)}\le \kappa<R
%  \end{equation}
% for a certain $M>0$, which will be a consequence of the energy estimates and
% which ensures that  $J_{\eta}$ and $J_{\eta}^{-1}$ remain bounded by a constant depending on $M$ and $\left|\nabla\eta\right|$.

Between the area elements of $\Gamma_{\eta}$ and $\Gamma$ we have that
$dA_{\eta}=J_{\eta}dA$.

% \begin{equation}
%  \left.\begin{array}{c}
% \rho_{s}h\partial_{tt}\eta+\Delta^{2}\eta=-J_{\eta\left(t\right)}\sigma\left(\phi_{\eta\left(t\right)}\left(x\right)\right)\boldsymbol{\nu}^{\eta}\left(t,x\right)\cdot\mathbf{e}_{r}\\
% \partial_{t}\eta\mathbf{e}_{r}\cdot\boldsymbol{\nu}^{\eta}\left(t,z\right)=\mathbf{u}\left(t,\phi_{\eta\left(t\right)}\left(x\right)\right)\cdot\boldsymbol{\nu}^{\eta}\left(t,x\right)\\
% \left(\partial_{t}\eta\mathbf{e}_{r}-\mathbf{u}\left(t,\phi_{\eta\left(t\right)}\left(x\right)\right)\right)\cdot\boldsymbol{\tau}_{i}^{\eta}=\alpha\sigma\left(\phi_{\eta\left(t\right)}\left(x\right)\right)\boldsymbol{\nu}^{\eta}\left(t,x\right)\cdot\boldsymbol{\tau}_{i}^{\eta}\ \left(i=1,2\right)
% \end{array}\right\} \text{in}\ \left(t,x\right)\in I\times\omega\end{equation}

% Concerning the elastic shell, we assume that it is \emph{clamped} at the endpoints, meaning that
% \begin{equation}
% \left|\eta\right|=\left|\nabla\eta\right|=0\quad\text{at}\ z\in\left\{ 0,L\right\} 
% \end{equation}

The displacement $\eta$ enjoys space periodic boundary conditions due to the use of $\omega=\mathbb{T}^{2}$.

\begin{remark}\label{rmk:int-eta-constant-time}
    Note that from the incompressibility condition we have
    \[
    \int_{\Omega_{\eta}\left(t\right)}\text{div} \ \mathbf{u}dx=0
    \]

   and  it follows that 
\[
\int_{\omega}\partial_{t}\eta dA=0\Longleftrightarrow\int_{\omega}\eta dA=:m,\quad\forall t\in I
\]
for some $m \in \mathbb{R}$.
\end{remark}

\begin{remark}\label{rmk:data}
For simplicity of the exposition we will refer to the $\texttt{data}$ of the problem as the set of prescribed constants and parameters:
\begin{equation}\label{eqn:data}
\texttt{data}:=\left\{ T,\,\Omega,\,\omega,\,\rho_{f},\,\mu_{f},\,\rho_{s},\,h,\,\lambda,\,\mu,\,\kappa\right\}.
\end{equation}
\end{remark}

\subsection{The fluid-structure interaction (FSI) problem} 
The equations presented above can be summarized in the following problem:
find a pair $\left(\mathbf{u},\eta\right)$ as described above which  solves the following system: 
\begin{equation}\label{eqn:FSI-system}
\begin{cases}
\partial_{t}\mathbf{u}+\left(\mathbf{u}\cdot\nabla\right)\mathbf{u}=\text{div}\ \sigma+\mathbf{f} & \text{in}\ I\times\Omega_{\eta}\\
\text{div}\mathbf{u}=0 & \text{in}\ I\times\Omega_{\eta}\\
\mathbf{u}=\mathbf{0} & \text{on}\ I\times\left\{ z=0\right\} \\
\partial_{tt}\eta+K^{\prime}\left(\eta\right)=-J_{\eta}\left(\sigma\left(\mathbf{u},p\right)\boldsymbol{\nu}^{\eta}\left(t,\phi_{\eta}\right)\right)\cdot\mathbf{e}_{3}+g & \text{on}\ I\times\omega\\
\eta\left(0,\cdot\right)=\eta\left(T,\cdot\right),\ \partial_{t}\eta\left(0,\cdot\right)=\partial_{t}\eta\left(T,\cdot\right) & \text{in}\ \omega\\
\mathbf{u}\left(0,\cdot\right)=\mathbf{u}\left(T,\cdot\right) & \text{in \ensuremath{\Omega_{\eta\left(0\right)}=}\ensuremath{\Omega_{\eta\left(T\right)}}}
\end{cases}\tag{FSI}
\end{equation}

\section{Formal a-priori estimates}\label{ssec:a-priori-section-paper}
Assume within this subsection that all the functions involved are smooth. 
We multiply the fluid equations \eqref{eqn:fluid} by $\mathbf{u}$ and the elastic equation \eqref{eqn:lame} by $\partial_{t}\eta$. We use the boundary conditions and add the results to obtain the following energy balance
\begin{equation}\label{eqn:energy-balance}
    \frac{d}{dt}E\left(t\right)+\int_{\Omega_{\eta}\left(t\right)}\left|\nabla\mathbf{u}\right|^{2}dx=\int_{\Omega_{\eta}\left(t\right)}\mathbf{f}\cdot\mathbf{u}dx+\int_{\omega}g\partial_{t}\eta dA,\quad\forall t\in I
\end{equation}
with $E$ denoting the energy of the system, namely
\begin{equation}\label{eqn:energy-def}
 E\left(t\right):=\frac{1}{2}\int_{\Omega_{\eta}\left(t\right)}\left|\mathbf{u}\left(t\right)\right|^{2}dx+\frac{1}{2}\int_{\omega}\left|\partial_{t}\eta\left(t\right)\right|^{2}dA+K\left(\eta\left(t\right)\right)
\end{equation}
Assuming that the unknowns are time we can integrate \eqref{eqn:energy-balance} on $[0,T]$ to obtain that 
\begin{equation}\label{eqn:diff-eq}
    \int_{I}\int_{\Omega_{\eta}\left(t\right)}\left|\nabla\mathbf{u}\right|^{2}dxdt=\int_{I}\int_{\Omega_{\eta}\left(t\right)}\mathbf{f}\cdot\mathbf{u}dxdt+\int_{I}\int_{\omega}g\partial_{t}\eta dAdt.
\end{equation}
Since $\mathbf{u}=\mathbf{0}$ when $z=0$ we have Poincar\'{e}'s inequality
\begin{equation}\label{eqn:poincare}
\int_{\Omega_{\eta}\left(t\right)}\left|\mathbf{u}\right|^{2}dx\lesssim\int_{\Omega_{\eta}\left(t\right)}\left|\nabla\mathbf{u}\right|^{2}dx
\end{equation}
where the constant depends on the diameter on $\Omega_{\eta}$. Using \eqref{eqn:condition-eta-norm}
we get  
\begin{equation}\label{eqn:y1}
\int_{\Omega_{\eta}\left(t\right)}\left|\mathbf{u}\right|^{2}dx\lesssim_{\Omega,\kappa}\int_{\Omega_{\eta}\left(t\right)}\left|\nabla\mathbf{u}\right|^{2}dx.
\end{equation}
Next, we write 
\begin{equation}\label{eqn:y2}
\begin{aligned}\int_{\omega}\left|\partial_{t}\eta\right|^{2}dA & =\int_{\omega}\left|\mathbf{u}\left(t,x,y,\eta\right)-\mathbf{u}\left(t,x,y,0\right)\right|^{2}dA\\
 & =\int_{\omega}\left|\int_{0}^{\eta\left(t,x,y\right)}\partial_{z}\mathbf{u}dz\right|^{2}dA\\
 & \le\int_{\omega}\eta\int_{0}^{\eta\left(t,x,y\right)}\left|\partial_{z}\mathbf{u}\right|^{2}dz\ dA\\
 & \le\left\Vert \eta\right\Vert _{L_{t,x}^{\infty}}\int_{\Omega_{\eta}\left(t\right)}\left|\nabla\mathbf{u}\right|^{2}dx\\
 & \le\int_{\Omega_{\eta}\left(t\right)}\left|\nabla\mathbf{u}\right|^{2}dx.
\end{aligned}
\end{equation}

We may now use Young's inequality in \eqref{eqn:diff-eq} and \eqref{eqn:y1}-\eqref{eqn:y2} to get that, for every $\varepsilon>0$ the following holds:
\begin{equation}
    \begin{aligned}\int_{I}\int_{\Omega_{\eta}\left(t\right)}\left|\nabla\mathbf{u}\right|^{2}dxdt\le & \varepsilon\left(\int_{I}\int_{\Omega_{\eta}\left(t\right)}\left|\mathbf{u}\right|^{2}dxdt+\int_{I}\int_{\omega}\left|\partial_{t}\eta\right|^{2}dAdt\right)+\\
 & \frac{1}{4\varepsilon}\underbrace{\left(\int_{I}\int_{\Omega_{\eta}\left(t\right)}\left|\mathbf{f}\right|^{2}dxdt+\int_{I}\int_{\omega}\left|g\right|^{2}dAdt\right)}_{C\left(\mathbf{f},g\right)}\\
 & \lesssim\varepsilon\int_{I}\int_{\Omega_{\eta}\left(t\right)}\left|\nabla\mathbf{u}\right|^{2}dxdt+\frac{1}{4\varepsilon}C\left(\mathbf{f},g\right)
\end{aligned}
\end{equation}
denoting
\begin{equation}\label{eqn:C(f,g)}
C\left(\mathbf{f},g\right):=\int_{I}\int_{\Omega_{\eta}\left(t\right)}\left|\mathbf{f}\right|^{2}dxdt+\int_{I}\int_{\omega}\left|g\right|^{2}dAdt
\end{equation}
and by choosing $\varepsilon$ sufficiently small we see that 
\begin{equation}\label{eqn:diffusion-estimate}
\begin{aligned}\int_{I}\int_{\Omega_{\eta}\left(t\right)}\left|\mathbf{u}\right|^{2}dxdt+\int_{I}\int_{\omega}\left|\partial_{t}\eta\right|^{2}dAdt\lesssim & \int_{I}\int_{\Omega_{\eta}\left(t\right)}\left|\nabla\mathbf{u}\right|^{2}dxdt\\
\lesssim & C\left(\mathbf{f},g\right).
\end{aligned}
\end{equation}

Now, our aim is to obtain an uniform in time estimate of the energy. For this, by the mean value Theorem there is $t_{0}\in I$ such that 
\[
E\left(t_{0}\right)=\frac{1}{\left|I\right|}\int_{I}E\left(t\right)dt=\frac{1}{T}\int_{I}E\left(t\right)dt.
\]
Integrating \eqref{eqn:energy-balance} between $t_0$ and an arbitrary $t\in I$ and employing \eqref{eqn:diffusion-estimate} and Young's inequality we obtain
\begin{equation}\label{eqn:unif-est-1}
    \begin{aligned}\sup_{t\in I}E\left(t\right)\le & E\left(t_{0}\right)+\int_{I}\int_{\Omega_{\eta}\left(t\right)}\left|\nabla\mathbf{u}\right|^{2}dxdt+\\
 & \frac{1}{2}\left(\int_{I}\int_{\Omega_{\eta}\left(t\right)}\left|\mathbf{u}\right|^{2}dxdt+\int_{I}\int_{\omega}\left|\partial_{t}\eta\right|^{2}dAdt\right)+\frac{1}{2}C\left(\mathbf{f},g\right)\\
\lesssim & \int_{I}K\left(\eta\right)dt+C\left(\mathbf{f},g\right).
\end{aligned}
\end{equation}

Now, in order to estimate the term $\int_{I}K\left(\eta\right)dt$ we would like to test the elastic equation \eqref{eqn:lame} by $\eta$. For this, we need to find a corresponding test-function for the fluid equations \eqref{eqn:fluid}. 

To this end, we employ the extension operator  $\mathcal{F}_{\eta}$ for the fluid equation such that $\text{div} \mathcal{F}(\eta)=0$ and $\text{tr}_{\eta}\mathcal{F}(\eta)=\eta \mathbf{e}_{3} $.  
Finding such an operator $\mathcal{F}(\eta)$ is not straight-forward. In \cite[Proposition 3.3]{MS22}, this construction was performed, along with corresponding  estimates. 
The extension is $\mathcal{F}=\mathcal{F}_{\eta}$, which is defined along with appropriate estimates in Proposition~\ref{chap-appendix-prop-extension-plate}.
Now let us fix a smooth compactly supported function
\begin{equation}\label{eqn:psi-bump-function-test-eta}
\psi\in C^{\infty}\left(\omega;\mathbb{R}\right)\quad \int_{\omega}\psi dA=1
\end{equation}
 and by recalling the definition of $m$ from Remark~\ref{rmk:int-eta-constant-time} we find  that the following couple 
\begin{equation}\label{eqn:test-funct-eta}
\left(\mathcal{F}_{\eta}\left(\mathcal{M}_{\eta}\left(\eta\right)\right),\mathcal{M}_{\eta}\left(\eta\right)\right)=\left(\mathcal{F_{\eta}}\left(\eta-m\psi\right),\eta-m\psi\right)
\end{equation}
is a valid test function. 
As we show below, this test-function  provides following estimate: 
\begin{equation}\label{eqn:formal-claim}
2\int_{I}K\left(\eta\right)dt\lesssim\theta\sup_{t\in I}E\left(t\right)+\frac{1}{\theta}\left(C\left(\mathbf{f},g\right)^{2}+C\left(\mathbf{f},g\right)+m^{2}\right)\quad\forall\theta>0.
\end{equation}

Let us recall  that we are in the case when the Koiter energy takes, up to some constants, the form \[K\left(\eta\right)=\int_{\omega}\left|\nabla\eta\right|^{4}+\left|\nabla^{2}\eta\right|^{2}dA\] and routine computations show that
 \[
\left\langle K^{\prime}\left(\eta\right),\xi\right\rangle =4\int_{\omega}\left|\nabla\eta\right|^{2}\nabla\eta\cdot\nabla\xi+2\nabla^{2}\eta:\nabla^{2}\xi,\quad\xi\in H^{2}\left(\omega\right)
 \] In particular we have that 
 \[\left\langle K^{\prime}\left(\eta\right),\eta\right\rangle \ge2K\left(\eta\right).\]

Using the test-function \eqref{eqn:test-funct-eta} we find that
\begin{equation}
\begin{aligned}2\int_{I}K\left(\eta\right)dt\leq & \int_{I}\left\langle K^{\prime}\left(\eta\right),\eta\right\rangle dt\\
\leq & \int_{I}\left\langle K^{\prime}\left(\eta\right),\eta-\mathcal{M}_{\eta}\left(\eta\right)\right\rangle dt+\int_{I}\int_{\omega}\partial_{t}\eta\partial_{t}\left(\mathcal{M}_{\eta}\left(\eta\right)\right)dAdt+\\
 & \int_{I}\int_{\Omega_{\eta\left(t\right)}}\mathbf{u}\cdot\partial_{t}\mathcal{F}_{\eta}\left(\mathcal{M}_{\eta}\left(\eta\right)\right)dxdt+\\
 & \int_{I}\int_{\Omega_{\eta\left(t\right)}}\left(\nabla\mathbf{u}+\mathbf{u}\otimes\mathbf{u}\right):\nabla\mathcal{F}_{\eta}\left(\mathcal{M}\left(\eta\right)\right)dxdt+\\
 & \int_{I}\int_{\Omega_{\eta\left(t\right)}}\mathbf{f}\cdot\mathcal{F}_{\eta}\left(\mathcal{M}\left(\eta\right)\right)dxdt+\int_{I}\int_{\omega}g\left(\mathcal{M}_{\eta}\left(\eta\right)\right)dAdt\\
  =:&\sum_{k=1}^{7}I_{k}.
\end{aligned}
\end{equation}

Recalling now the properties of the divergence-free extension operator $\mathcal{F}_{\eta}$ from Proposition~\ref{chap-appendix-prop-extension-plate}
we employ  H\"{o}lder's inequalilty,  the Sobolev embedding, the $H^2$-coercivity of $K(\eta)$ and the estimate~\eqref{eqn:C(f,g)} to get 
\begin{equation}\label{eqn:chap02-estimate-K-eta-discrete}
\begin{aligned}\left|I_{1}\right| & \le m\left\Vert K^{\prime}\left(\eta\right)\right\Vert _{L_{t}^{2}L_{x}^{2}}\le m\left(\sup_{t\in I}E\left(t\right)\right)^{1/2}\\
\left|I_{2}\right| & \le\left\Vert \partial_{t}\eta\right\Vert _{L_{t}^{2}L_{x}^{2}}^{2}\lesssim C\left(\mathbf{f},g\right)\\
\left|I_{3}\right|+\left|I_{4}\right|+\left|I_{5}\right| & \lesssim\left(\left\Vert \mathbf{u}\right\Vert _{L_{t}^{2}W_{x}^{1,2}}+\left\Vert \mathbf{u}\right\Vert _{L_{t}^{2}W_{x}^{1,2}}^{2}\right)\left\Vert \mathcal{F}_{\eta}\left(\mathcal{M}_{\eta}\right)\right\Vert _{W_{t}^{1,2}L_{x}^{2}\cap L_{t}^{\infty}W_{x}^{1,2}}\\
 & \lesssim\left(C\left(\mathbf{f},g\right)^{1/2}+C\left(\mathbf{f},g\right)\right)\left\Vert \mathcal{M}_{\eta}\right\Vert _{W_{t}^{1,\infty}L_{x}^{2}\cap L_{t}^{\infty}W_{x}^{2,2}}\\
 & \lesssim\left(C\left(\mathbf{f},g\right)^{1/2}+C\left(\mathbf{f},g\right)\right)\left(\left(\sup_{t\in I}E\left(t\right)\right)^{1/2}+m\right)\\
\left|I_{6}\right|+\left|I_{7}\right| & \le\left(\left\Vert \mathbf{f}\right\Vert _{L_{t}^{2}L_{x}^{2}}+\left\Vert g\right\Vert _{L_{t}^{2}L_{x}^{2}}\right)\left(\left\Vert \mathcal{F}_{\eta}\left(\mathcal{M}_{\eta}\right)\right\Vert _{L_{t}^{2}L_{x}^{2}}+\left\Vert \mathcal{M}_{\eta}\right\Vert _{L_{t}^{2}L_{x}^{2}}\right)\\
 & \lesssim C\left(\mathbf{f},g\right)^{1/2}\left(\sup_{t\in I}E\left(t\right)\right)^{1/2}+C\left(\mathbf{f},g\right)
\end{aligned}
\end{equation}

All in all,  using Young's inequality
$
\left|ab\right|\le\theta a^{2}+\frac{1}{4\theta}b^{2}\quad\forall a,b\in\mathbb{R},\theta>0$
 we conclude that 
\[
\sum_{k=1}^{7}\left|I_{k}\right|\lesssim\theta\sup_{t\in I}E\left(t\right)+\frac{1}{\theta}\left(C\left(\mathbf{f},g\right)^{2}+C\left(\mathbf{f},g\right)+m^{2}\right).
\]
From \eqref{eqn:formal-claim} and \eqref{eqn:unif-est-1}, choosing $\theta$ sufficiently small 
we obtain that
 \begin{equation}\label{eqn:formal-final}
\sup_{t\in I}E\left(t\right)\lesssim C\left(\mathbf{f},g\right)^{2}+C\left(\mathbf{f},g\right)+m^{2}.
\end{equation}
This is the uniform estimate of energy which is needed in the context of time- periodic solutions.

\section{Weak formulation and main result}\label{ssec:defweaksol}

% Motivated by the a-priori estimates we can introduce the following  notion of \emph{periodic solution}. 
We derive formally a weak formulation of the problem~\eqref{eqn:FSI-system}.

For a generic Banach space $X$ let us consider the space 
\[
C_{\text{per}}^{\infty}\left(\mathbb{R};X\right):=\left\{ \varphi\in C^{\infty}\left(\mathbb{R};X\right),\ \varphi\left(t\right)=\varphi\left(t+T\right)\ \forall t\in\mathbb{R}\right\} 
\]
and let 
\[
\left(\mathbf{q},\xi\right)\in C_{\text{per}}^{\infty}\left(I;C^{\infty}\left(\Omega_{\eta}\right)\right)\times C_{\text{per}}^{\infty}\left(I;C^{\infty}\left(\omega\right)\right)
\]
be such that 
\begin{equation}
\text{tr}_{\eta}\left(\mathbf{q}\right):=\left.\mathbf{q}\right|_{z=\eta\left(t,x,y\right)}=\xi\mathbf{e_{3}},\ \text{div}\ \mathbf{q}=0,\quad\left.\mathbf{q}\right|_{z=0}=\mathbf{0}
\end{equation}

The pair  $(\mathbf{q}, \xi )$  as above will be a test function. We multiply \eqref{eqn:fluid} by $\mathbf{q}$ and \eqref{eqn:lame} by $\xi$. 

First we notice that
\begin{align*}\int_{I}\int_{\Omega_{\eta\left(t\right)}}\partial_{t}\mathbf{u}\cdot\mathbf{q}dxdt= & \int_{I}\left(\frac{d}{dt}\int_{\Omega_{\eta\left(t\right)}}\mathbf{u}\cdot\mathbf{q}dx-\int_{\Omega_{\eta\left(t\right)}}\mathbf{u}\cdot\partial_{t}\mathbf{q}dx\right)dt\\
 & -\int_{I}\int_{\omega}\left(\partial_{t}\eta\right)^{2}\xi dAdt\\
= & -\int_{I}\int_{\Omega_{\eta\left(t\right)}}\mathbf{u}\cdot\partial_{t}\mathbf{q}dxdt-\int_{I}\int_{\omega}\left(\partial_{t}\eta\right)^{2}\xi dAdt
\end{align*}
 where we have  used Reynolds' transport Theorem~\ref{thm:Reynolds-chap-appendix}  and the fact that $\text{tr}_{\eta} \mathbf{u}=\partial_t\eta \mathbf{e}_{3}$ and also $\text{tr}_{\eta} \mathbf{q}= \xi \mathbf{e}_3$. On the other hand if we look at the convective term we see that (using Einstein's summation convention)
 \begin{equation}
\begin{aligned}\int_{I}\int_{\Omega_{\eta\left(t\right)}}\left(\mathbf{u}\cdot\nabla\right)\mathbf{u}\cdot\mathbf{q}dxdt= & \int_{I}\int_{\Omega_{\eta\left(t\right)}}\partial_{i}\left(\mathbf{u}^{i}\mathbf{u}^{j}\right)\mathbf{q}^{j}-\partial_{i}\mathbf{u}^{i}\mathbf{u}^{j}\mathbf{q}^{j}dxdt\\
= & \int_{I}\int_{\partial\Omega_{\eta\left(t\right)}}\mathbf{u}^{i}\mathbf{u}^{j}\mathbf{q}^{j}\nu^{i}dAdt\\
 & -\int_{I}\int_{\Omega_{\eta\left(t\right)}}\left(\mathbf{u}^{i}\mathbf{u}^{j}\right)\partial_{i}\mathbf{q}^{j}dxdt\\
= & \int_{I}\int_{\omega}\left(\partial_{t}\eta\right)^{2}\xi dAdt-\int_{I}\int_{\Omega_{\eta\left(t\right)}}\left(\mathbf{u}\cdot\nabla\right)\mathbf{q}\cdot\mathbf{u}dxdt
\end{aligned}
 \end{equation}
where we have used that $\text{tr}_{\eta}\mathbf{u}=\partial_{t}\eta\mathbf{e}_{3}$,  $\text{div} \ \mathbf{u}=0$ and Stokes' Theorem. 

Recalling  Korn's identity  from Lemma~\ref{lm:Korn} we get
$\int_{\Omega_{\eta\left(t\right)}}\nabla\mathbf{u}:\nabla\mathbf{q}dx=2\int_{\Omega_{\eta\left(t\right)}}\mathbb{D}\left(\mathbf{u}\right):\mathbb{D}\left(\mathbf{q}\right)dx$.
Putting all together,  we obtain the following definition. 

Now we need to introduce the suitable function spaces.
\paragraph{Function spaces.}

Recalling that $I = [0,T]$, we define
\begin{equation}
I \times \Omega_{\eta} := \bigcup_{t\in I} \left\{ t\right\} \times\Omega_{\eta\left(t\right)}
\end{equation} and we can now introduce the Lebesgue and Sobolev spaces adapted to the moving domains. They are defined as follows: for $1 \le p, \ q \le \infty$ we have
\begin{equation}
\begin{aligned}L^{p}\left(I;L^{q}\left(\Omega_{\eta}\right)\right):= & \left\{ \mathbf{v}\in L^{1}\left(I\times\Omega_{\eta}\right):\mathbf{v}\left(t,\cdot\right)\in L^{q}\left(\Omega_{\eta\left(t\right)}\right)\ \text{for a .e.}\ t\in I,\right.\\
 & \left\Vert \mathbf{v}\left(t,\cdot\right)\right\Vert _{L^{q}\left(\Omega_{\eta\left(t\right)}\right)}\in L^{p}\left(I\right)\left.\right\} \\
L^{p}\left(I;W^{1,q}\left(\Omega_{\eta}\right)\right):= & \left\{ \mathbf{v}\in L^{1}\left(I\times\Omega_{\eta}\right):\mathbf{v}\left(t,\cdot\right)\in W^{1,q}\left(\Omega_{\eta\left(t\right)}\right)\ \text{for a .e.}\ t\in I,\right.\\
 & \left\Vert \nabla\mathbf{v}\left(t,\cdot\right)\right\Vert _{L^{q}\left(\Omega_{\eta\left(t\right)}\right)}\in L^{p}\left(I\right)\left.\right\} 
\end{aligned}
\end{equation}

These spaces admit  the  usual \emph{trace} and \emph{extension} operators as long as  $\eta\in C_{x}^{0,1}$ and $\left\Vert \eta\right\Vert _{L_{t,x}^{\infty}}<1$; see  \cite[Lemma 1]{Gr05} for further details.
We will write for simplicity $W_{t}^{k,p}W_{x}^{l,q}$ when dealing with Lebesgue/Sobolev spaces of Bochner type.

%\subsection{Spaces of time-periodic functions}
In the spirit of the  formal estimates that will be derived, it  is natural to consider the following function spaces. 
 For  $T>0$ and recalling that
 \[
C_{\text{per}}^{\infty}\left(I;X\right):=\left\{ \varphi\in C^{\infty}\left(\mathbb{R};X\right),\ \varphi\left(t+T,\cdot\right)=\varphi\left(t,\cdot\right)\right\} 
 \] 
 
 for $k \in \mathbb{Z}$ and $1\le p\le \infty$  we may define the  spaces
\begin{align*}
L_{\text{per}}^{p}\left(I;X\right):= & \overline{\left\{ f\in C_{\text{per}}^{\infty}\left(I;X\right)\right\} }^{\left\Vert \cdot\right\Vert _{p}}\\
W_{\text{per}}^{k,p}\left(I;X\right):= & \overline{\left\{ f\in C_{\text{per}}^{\infty}\left(I;X\right)\right\} }^{\left\Vert \cdot\right\Vert _{k,p}}.
\end{align*}
 
Now we introduce the function spaces adapted to the solutions and the test functions. 
  Such spaces have been investigated carefully in the previous literature (see for instance~\cite{LR14}) and hence the concept of time-periodic spaces can be extended in a straight forward manner. 
Let us set
\begin{equation}\label{eqn:def-function-spaces}
\begin{aligned}H_{0,\text{div}}^{1}\left(\Omega_{\eta}\right):= & \left\{ \mathbf{u}\in H^{1}\left(\Omega_{\eta}\right):\text{div}\ \mathbf{u}=0\text{ in }\Omega_{\eta},\ \left.\mathbf{u}\right|_{z=0}=\mathbf{0}\right\} \\
V_{Fl,\text{per}}:= & L_{\text{per}}^{\infty}\left(I;L^{2}\left(\Omega_{\eta}\right)\right)\cap L_{\text{per}}^{2}\left(I;H_{0,\text{div}}^{1}\left(\Omega_{\eta}\right)\right)\\
V_{St,\text{per}}:= & L_{\text{per}}^{\infty}\left(I;H^{2}\left(\omega\right)\right)\cap W_{\text{per}}^{1,\infty}\left(I;L^{2}\left(\omega\right)\right)\\
V_{S,\text{per}}:= & \left\{ \left(\mathbf{u},\eta\right)\in V_{F,\text{per}}\times V_{K,\text{per}};\mathbf{u}\left(t,\phi_{\eta}\left(t,x\right)\right)=\partial_{t}\eta\left(t,x\right)\text{\ensuremath{\mathbf{e}_{3}} in }\omega\right\} \\
V_{T,Fl} & :=H_{\text{per}}^{1}\left(I;L^{2}(\Omega_{\eta}\right)\cap L_{\text{per}}^{\infty}\left(I;H_{0,\text{div}}^{1}\left(\Omega_{\eta}\right)\right)\\
V_{T,St} & :=\left\{ \xi:\xi\in C_{\text{per}}^{\infty}\left(I;C^{\infty}\left(\omega;\mathbb{R}\right)\right)\right\} \\
V_{T,\text{per}}:= & \left\{ \left(\mathbf{q},\xi\right)\in V_{T,Fl}\times V_{T,St}:\mathbf{q}\left(t,\phi_{\eta}\left(t,x\right)\right)=\xi\left(t,x\right)\mathbf{e}_{3}\ \text{in}\ \omega\right\} 
\end{aligned}
\end{equation}
Here $V_{S,\text{per}}$ will be the space of time-periodic solutions and $V_{T,\text{per}}$  the space of time-periodic test-functions.

\begin{remark}\label{rmk:bound-eta-norm}
    In order to estimate $\left\Vert \eta\right\Vert _{H^{2}\left(\omega\right)}$ let us observe that from \eqref{eqn:formal-final} one gets the bound 
    \begin{equation}
        \sup_{t\in I}\left(\left\Vert \nabla^{2}\eta\right\Vert _{L^{2}\left(\omega\right)}^{2}+\left\Vert \nabla\eta\right\Vert _{L^{4}\left(\omega\right)}^{4}\right)\lesssim C\left(\mathbf{f},g\right)^{2}+C\left(\mathbf{f},g\right)+m^{2},\quad m=\int_{\omega}\eta dA.
    \end{equation}
    Using Poincar\'{e}'s inequality we further obtain 
    \begin{equation}
        \left\Vert \eta-m\right\Vert _{L^{2}\left(\omega\right)}\lesssim\left\Vert \nabla\eta\right\Vert _{L^{2}\left(\omega\right)}\lesssim\left(C\left(\mathbf{f},g\right)^{2}+C\left(\mathbf{f},g\right)+m^{2}\right)^{1/4}
    \end{equation}
and thus we  bound $\left\Vert \eta\right\Vert _{L_{t}^{\infty}H^{2}\left(\omega\right)}$ by $C\left(\mathbf{f},g\right)^{2}+C\left(\mathbf{f},g\right)+m^{2}$. In particular when $C\left(\mathbf{f},g\right),m\to0$ it will follow that $\left\Vert \eta\right\Vert _{L_{t}^{\infty}L^{\infty}\left(\omega\right)}\lesssim\left\Vert \eta\right\Vert _{L_{t}^{\infty}H^{2}\left(\omega\right)}\to0$ hence can be made arbitrarly small, as in \eqref{eqn:condition-eta-norm}.
    
\end{remark}

\begin{remark}[On the notations]
Further, if the geometry is {\em decoupled} from the solid deformation we shall use the analogous notations $V_{T,\text{per}}^{\delta},V_{S,\text{per}}^{\delta}$  whenever we deal with functions $(u, \eta)$ for which 
$
\text{tr}_{\delta}(\mathbf{u})=\mathbf{u}\left(t,\phi_{\delta(t)}\right)=\partial_{t}\eta\nu $ for a given function $\delta\in C\left(I\times\omega\right)$ with $\ \left\Vert \delta\right\Vert _{L^{\infty}\left(I\times\omega\right)}<\kappa$ prescribing a given time-changing domain.
If we omit to mention the term \emph{per} we refer to the closure of the  smooth functions which are not necessarily time-periodic. This might be the case when the functions are not yet proven to be time-periodic and thus merely $\left(\mathbf{u},\eta\right)\in V_{S}$ for example. 
\end{remark}
\begin{remark} \label{remark: Holder-cont}
We have $V_{K,\text{per}}\hookrightarrow C^{0,1-\theta}\left(\overline{I};C^{0,2\theta-1}\left(\omega\right)\right)$, and therefore  the displacement of the shell is H\"{o}lder continuous in time.
 Indeed, using the embeddings $H^{2}\left(\omega\right)\hookrightarrow H^{2\theta}\left(\omega\right)\hookrightarrow C^{0,2\theta-1}\left(\omega\right)$ for $\theta \in (1/2,1)$  (since $\omega\subset \mathbb{R}^2$) we find that
\begin{align*}\left\Vert \eta\left(t\right)-\eta\left(s\right)\right\Vert _{C^{0,2\theta-1}\left(\omega\right)} & \lesssim \norm{\eta\left(t\right)-\eta\left(s\right)}_{H^{2\theta}(\omega)}\\
 & \lesssim \left\Vert \eta\left(t\right)-\eta\left(s\right)\right\Vert _{H^{2}\left(\omega\right)}^{\theta}\left\Vert \eta\left(t\right)-\eta\left(s\right)\right\Vert _{L^{2}\left(\omega\right)}^{1-\theta}\\
 & \lesssim\left\Vert \eta\right\Vert _{L^{\infty}\left(I;H^{2}\left(\omega\right)\right)}^{\theta}\left\Vert \eta\right\Vert _{W^{1,\infty}\left(I;L^{2}\left(\omega\right)\right)}\left|t-s\right|^{1-\theta}.
\end{align*}
In particular, the condition $\eta(0,\cdot)=\eta(T,\cdot)$   holds in a strong sense, in contrast with   the  conditions  $\mathbf{u}(0,\cdot)=\mathbf{u}(T,\cdot)$ and $\partial_{t} \eta(0,\cdot)=\partial_{t} \eta (T,\cdot)$ which  can only occur in a weak  sense.
\end{remark}

\begin{definition}\label{def:weak-periodic-solution}
Let $m\in \mathbb{R}$.
A couple $(\mathbf{u},\eta) \in V_{S, \text{per}}$ with $\left\Vert \eta\right\Vert _{L^{\infty}\left(I\times\omega\right)}<1$ and $\int_{\omega}\eta dA=m$
is called a \emph{time-periodic weak solution} to \eqref{eqn:FSI-system} provided that:
\begin{enumerate}
    \item The following relation holds:
\begin{equation}\label{eq: weak-periodic-solution}
\begin{aligned}\int_{I}\int_{\Omega_{\eta\left(t\right)}}-\mathbf{u}\cdot\partial_{t}\mathbf{q}+\nabla\mathbf{u}:\nabla\mathbf{q}-\left(\mathbf{u}\cdot\nabla\right)\mathbf{q}\cdot\mathbf{u}dxdt & +\\
\int_{I}\left\langle K^{\prime}\left(\eta\right),\xi\right\rangle dt -\int_{I}\int_{\omega}\partial_{t}\eta\partial_{t}\xi dAdt & =\\
\int_{I}\int_{\Omega_{\eta\left(t\right)}}\mathbf{f}\cdot\mathbf{q}dxdt+\int_{I}\int_{\omega}g\xi dAdt\quad\forall\left(\mathbf{q},\xi\right)\in V_{T,\text{per}}=V_{T,\text{per}}^{\eta}.
\end{aligned}
\end{equation}

\item The following energy estimate holds:
 \begin{equation}\label{eqn:energy-estimate}
\sup_{t\in I}E\left(t\right)\lesssim C\left(\mathbf{f},g\right)^{2}+C\left(\mathbf{f},g\right)+m^{2}.
\end{equation}
with
\begin{equation}
E\left(t\right):=\frac{1}{2}\int_{\Omega_{\eta}\left(t\right)}\left|\mathbf{u}\left(t\right)\right|^{2}dx+\frac{1}{2}\int_{\omega}\left|\partial_{t}\eta\left(t\right)\right|^{2}dA+K\left(\eta\left(t\right)\right).
\end{equation}

\end{enumerate}
\end{definition}

\subsection{The main result}\label{ssec:main-result}
% Let us denote for a Banach space $X$ and $p\ge 1$ the space 
% \[L_{\text{per}}^{p}\left(I;X\right):=\left\{ \varphi\in C^{\infty}\left(\mathbb{R};X\right):\varphi\left(t\right)=\varphi\left(t+T\right)\ \forall t\in I\right\} ^{\left\Vert \cdot\right\Vert _{L^{p}\left(I:X\right)}}.\]

The main theorem of this work consists in the following
\begin{theorem}\label{thm:main} 

There exists a constant $C_{0}=C_{0}(\texttt{data})$    such that  for any $m \in \mathbb{R}$ and any given forces
\[
\left(\mathbf{f},g\right)\in L_{\text{per}}^{2}\left(I;L^{2}\left(\mathbb{R}^{3}\right)\right)\times L_{\text{per}}^{2}\left(I;L^{2}\left(\omega\right)\right)
\]
such that 
\[ m^2+
\int_{I}\int_{\mathbb{R}^{3}}\left|\mathbf{f}\right|^{2}dxdt+\int_{I}\int_{\omega}\left|g\right|^{2}dAdt\le C_0
\]
 the problem \eqref{eqn:FSI-system}  admits at least one  time-periodic weak solution $(\mathbf{u},\eta)$ in the sense of Definition~\ref{def:weak-periodic-solution}, with
\begin{equation}
\int_{\omega}\eta(t,\cdot)dA=m \quad \forall t\in I.
\end{equation} 

and satisfying 
\begin{equation}
    \sup_{t\in I}E\left(t\right)\lesssim C_{0}^{2}+C_{0}
\end{equation}
up to a constant depending on $\texttt{data}$ only. 
In particular it holds that \begin{equation}
\left\Vert \mathbf{u}\right\Vert _{L^{\infty}\left(I;L^{2}\left(\Omega_{\eta}\right)\right)\cap L^{2}\left(I;W_{\text{div}}^{1,2}\left(\Omega_{\eta}\right)\right)}+\left\Vert \eta\right\Vert _{L^{\infty}\left(I;H_{}^{2}\left(\omega\right)\right)\cap W^{1,\infty}\left(I;L^{2}\left(\omega\right)\right)}\lesssim C_{0}^{2}+C_{0}.
\end{equation}

Additionally,  the following  \emph{diffusion estimate} holds
\begin{equation}\label{eqn:energy-estimate-2-chap02}
\int_{I}\int_{\Omega_\eta\left(t\right)}\left|\nabla \mathbf{u}\right|^{2}dxdt\le\int_{I}\int_{\Omega_\eta\left(t\right)}\mathbf{f}\cdot \mathbf{u}dxdt+\int_{I}\int_{\omega}g\cdot\partial_{t}\eta dAdt.
\end{equation}
\end{theorem}

\begin{remark}[The 2D/1D case]
Our proof can immediately be adapted to the 2D-1D setting, that is of a  fluid-beam  system. 
\end{remark}

\section{Proof of the main result}\label{sec:proof}

The proof proceeds in three steps: (i) a Galerkin discretization on a \emph{decoupled}, prescribed-geometry domain $\Omega_\delta$; (ii) a fixed-point argument via the Leray-Schauder theorem to recover the coupling $\delta = \eta$; and (iii) passage to the limit in the discretization parameter $n\to\infty$ and in the regularization parameters $\varepsilon,\sigma\to 0$.

\subsection{Decoupled problem and Galerkin basis}
In the following we assume that a sufficiently smooth geometry $\delta$ is \emph{given}, for example
\[
\delta\in C_{\text{per}}^{3}\left(I;C^{3}\left(\omega\right)\right),\quad\left\Vert \delta\right\Vert _{L^{\infty}\left(I\times\omega\right)}<\kappa.
\]
The mapping $\delta$ is prescribing  time-changing domains $\Omega_{\delta}$. We introduce the decoupled weak formulation: we call $(\mathbf{u},\eta)\in V_{S, \text{per}}^{\delta}$ is a time-periodic weak solution for the \emph{decoupled} problem if
\begin{equation}\label{eq:dec-reg-per-wf}
\begin{aligned}\int_{I}\int_{\Omega_{\delta\left(t\right)}}-\mathbf{u}\cdot\partial_{t}\mathbf{q}+\nabla\mathbf{u}:\nabla\mathbf{q}+\frac{1}{2}\left(\mathbf{u}\cdot\nabla\right)\mathbf{u}\cdot\mathbf{q}dxdt & +\\
\int_{I}\int_{\Omega_{\delta\left(t\right)}}-\frac{1}{2}\left(\mathbf{u}\cdot\nabla\right)\mathbf{q}\cdot\mathbf{u}dxdt+\int_{I}\left\langle K^{\prime}\left(\eta\right),\xi\right\rangle dt & +\\
\int_{I}\int_{\omega}-\frac{1}{2}\left(\partial_{t}\eta\right)\left(\partial_{t}\delta\right)\xi+\partial_{t}\eta\partial_{t}\xi dAdt & =\\
\int_{I}\int_{\Omega_{\delta\left(t\right)}}\mathbf{f}\cdot\mathbf{q}dxdt+\int_{I}\int_{\omega}g\xi dAdt
\end{aligned}
\end{equation}
for all $(\mathbf{q},\xi)\in V_{T,\text{per}}^{\delta}$. The convective term is re-written in the skew-symmetric form to ensure that the energy balance 
\begin{equation}\label{eqn:energy-1-dec}
\frac{d}{dt}E\left(t\right)+\int_{\Omega_{\delta\left(t\right)}}\left|\nabla\mathbf{u}\right|^{2}dx=\int_{\Omega_{\delta\left(t\right)}}\mathbf{f}\cdot\mathbf{u}dx+\int_{\omega}g\partial_{t}\eta dA
\end{equation}
holds at the discrete level.

We now construct the Galerkin basis. Since $\omega=\mathbb{T}^{2}$ there is no boundary condition on $\partial\omega$ and the elastic plate displacement belongs to $H^{2}(\omega)$. Let $(\hat{Y}_{k})_{k \ge 1}$ be a smooth basis of $\left\{ \hat{Y}\in H^{2}\left(\omega\right):\ \int_{\omega}\hat{Y}dA=0\right\} $ which is orthogonal in $L^{2} (\omega)$.  
We  extend each $\hat{Y}_k$ to a divergence-free function $\mathbf{Y}_k : \Omega \mapsto \mathbb{R}^{3}$ by using the extension operator from Proposition~\ref{chap-appendix-prop-extension-plate}.
% solving the Stokes problem \[\begin{cases}
% -\Delta \mathbf{U}_{k}+\nabla P_{k}=0 & \text{in}\ \Omega\\
% \text{div} \ \mathbf{U}_{k}=0 & \text{in}\ \Omega\\
% \mathbf{U}_{k}=\hat{Y}_{k}\mathbf{e}_3& \text{on}\ \partial\Omega
% \end{cases}\]
% Note that here we need the $C^3$ regularity of $\Omega$ in order to use the results of  Theorem~ \ref{chap05-thm-Stokes-reg}.
%and of \cite{GSS94}.
%Then we use the Piola transform from Lemma~ \ref{lm:Piola} and define 
\begin{equation}
\label{eq:Y}
\left(\mathbf{Y}_{k},Y_{k}\right):=\left(\mathcal{F}_{\delta}\hat{Y}_{k},\hat{Y}_{k}\mathbf{e}_{3}\right)\in V_{T,\text{per}}^{\delta}
\end{equation}
We consider also $(\hat{Z}_{k})_{k \ge 1}$ a smooth,  $L^{2}$-orthonormal basis of the space $H^{1}_{0, \text{div}}(\Omega)$ which we extend using the Piola transform from Lemma~ \ref{lm:Piola}
to \begin{equation}
\left(\mathbf{Z}_{k}\left(t,\cdot\right),Z_{k}\left(t,\cdot\right)\right)=\left(\mathcal{J}_{\mathcal{\delta}\left(t\right)}\hat{Z}_{k},0\right)\in V_{T,\text{per}}^{\delta}.
\end{equation}
One can note that $\left\{ \mathbf{Z}_{k}\left(t,\cdot\right)\right\} _{k\ge1}$ forms a basis of $H_{0,\text{div}}^{1}\left(\Omega_{\delta\left(t\right)}\right)$ for each $t\in I$

Let us finally define \begin{equation}\label{eqn:defn-Xk} 
\left(\mathbf{X}_{k}\left(t,\cdot\right),X_{k}\left(t,\cdot\right)\right):=\begin{cases}
\left(\mathbf{Y}_{(k+1)/2}\left(t,\cdot\right),Y_{(k+1)/2}\left(t,\cdot\right)\right) & k\ \text{odd}\\
\left(\mathbf{Z}_{k/2}\left(t,\cdot\right),0\right) & k\ \text{even}
\end{cases}\quad\text{for all }k\in\mathbb{N}
\end{equation}
\begin{remark}\label{remark:density-in-testfunctions}
The space 
\[
\text{span}\left\{ \left(\varphi\mathbf{X}_{k},\varphi X_{k}\right):\varphi\in C_{\text{per}}\left[0,T\right]\cap C^{1}\left(0,T\right),k\in\mathbb{N}\right\} 
\]
 is dense in the space of test functions $V_{T, \text{per}} ^{\delta}$. 
 Indeed, let  $\left(\mathbf{q},\xi\right)\in V_{T,\text{per}}^{\delta}$ be smooth.
 We approximate $\xi$ by $\tilde \xi=\sum_{k=1}^{n}a_{k}Y_{k}$ in $H^{1}_{t}H^{2}_{x}$ and
 we have $\left(\mathbf{q},\xi\right)\sim\left(\overline{\mathbf{\xi}},\tilde{\xi}\right)+\left(\mathbf{q-}\overline{\mathbf{\xi}},0\right)$ where $\overline{\xi}:=\sum_{k=1}^{n}a_{k}\mathbf{Y}_{k}$. The quantity $\mathbf{q-}\overline{\mathbf{\xi}}$ is divergence free and has trace equal to zero, and thus can be approximated by a linear combination of vectors $\mathbf{Z}_{k}$ in $H^{1}_{t,x}$.
 \end{remark} 
We now make the following ansatz:
\begin{equation}\label{eqn:ansatz-u_n-eta_n}
\begin{aligned}\delta_{n}= & \sum_{k=1}^{n}a_{n}^{k}\left(t\right)X_{k}+m\psi,\quad\left(t,x\right)\in I\times\omega\\
\eta_{n}\left(t,x\right)= & \sum_{k=1}^{n}b_{n}^{k}\left(t\right)X_{k}+m\psi,\quad\left(t,x\right)\in I\times\omega\\
\mathbf{u}_{n}\left(t,x\right)= & \sum_{k=1}^{n}\left(b_{n}^{k}\right)^{\prime}\left(t\right)\mathbf{X}_{k}\left(t,x\right),\quad\left(t,x\right)\in I\times\Omega_{\delta}
\end{aligned}
\end{equation}
The following observation is important:  we will understand \eqref{eqn:ansatz-u_n-eta_n} as follows: 
\emph{given 
\[\mathbf{a}_{n}\in C^{1}\left(I;\mathbb{R}^{n}\right)\]
 and therefore given $\delta_n$ we seek for the unknown   
 \[\mathbf{b}_{n}\in C^{1}\left(I;\mathbb{R}^{n}\right)\]
  which is the solution of the  system of differential equations generated by \eqref{eqn:xk}.
  }

Note that we have ensured that the coupling condition is fulfilled, that is  \[\text{tr}_{\delta}(\mathbf{u}_{n})=\partial_{t}\eta_{n}\mathbf{e}_3.
\]

Due to the estimate \eqref{eqn:y2} (see also the trace operator and Lemma~\ref{lm:trace}) we get that
\begin{equation}\label{eqn:trace(eta_n)}
\int_{\omega}\left|\partial_{t}\eta_{n}\left(t\right)\right|^{2}dA\lesssim\int_{\Omega_{\mathcal{\delta}\left(t\right)}}\left|\nabla\mathbf{u}_{n}\left(t\right)\right|^{2}dx\quad\forall t\in I.
\end{equation}
We now  seek for 
\[
\mathbf{b}_{n}:=\left(\mathbf{b}_{n}^{k}\right)_{k=\overline{1,n}}:I\mapsto\mathbb{R}^{n}
\]
such that for each $1\le k\le n$ the following equation is satisfied:

\begin{equation}\label{eqn:xk}
\begin{aligned}\int_{\Omega_{\delta\left(t\right)}}\partial_{t}\mathbf{u}_{n}\cdot\mathbf{X}_{k}dx+\frac{1}{2}\int_{\omega}\partial_{t}\eta_{n}\partial_{t}\delta X_{k}dA+\int_{\Omega_{\delta\left(t\right)}}\nabla\mathbf{u}_{n}\cdot\nabla\mathbf{X}_{k}dx & +\\
\frac{1}{2}\int_{\Omega_{\delta\left(t\right)}}\left(\mathbf{u}_{n}\cdot\nabla\right)\mathbf{u}_{n}\cdot\mathbf{X}_{k}dx-\frac{1}{2}\int_{\Omega_{\delta\left(t\right)}}\left(\mathbf{u}_{n}\cdot\nabla\right)\mathbf{X}_{k}\cdot\mathbf{u}_{n}dx & +\\
\int_{\omega}\partial_{tt}\eta_{n}X_{k}dA+\left\langle K^{\prime}\left(\eta_{n}\right),X_{k}\right\rangle  & =\\
\int_{\Omega_{\delta\left(t\right)}}\mathbf{f}\cdot\mathbf{X}_{k}dx+\int_{\omega}gX_{k}dA\quad\forall t\in I
\end{aligned}
\end{equation}
Please note  that   \eqref{eqn:xk} reads as a second order system of  differential equations with the unknown $\mathbf{b}_n$. Now we prove that for any given initial values $\mathbf{b}_{n}\left(0\right),\mathbf{b}_{n}^{\prime}\left(0\right)$ this system \eqref{eqn:xk} has a unique solution. 
Indeed,  the coefficient of $\left(\mathbf{b}_{n}\right)^{\prime\prime}$ is given by the mass matrix
\[
M\left(t\right):=\left(\int_{\Omega_{\delta\left(t\right)}}\mathbf{X}_{i}\cdot\mathbf{X}_{j}dx\right)_{1\le i,j\le n}+I_{n}=:\tilde{M}\left(t \right)+I_n.
\]
Since for any $\xi \in \mathbb{R}^{n}$ we have that 
\begin{equation}\label{eqn:M(t)-pos-def}
\begin{aligned}M\left(t\right)\xi\cdot\xi & =\tilde{M}\left(t\right)\xi\cdot\xi+\text{diag}\left(1,0,1,0,\ldots\right)\xi\cdot\xi\\
 & \ge\sum_{i,j=1}^{n}\xi_{i}\xi_{j}\int_{\Omega_{\delta\left(t\right)}}\mathbf{X}_{i}\cdot\mathbf{X}_{j}dx\\
 & \ge\int_{\Omega_{\delta\left(t\right)}}\left(\sum_{i=1}^{n}\xi_{i}\mathbf{X}_{i}\right)\cdot\left(\sum_{j=1}^{n}\xi_{j}\mathbf{X}_{j}\right)dx\\
 & \ge0
\end{aligned}
\end{equation}
it follows that $M\xi=0$ implies that $\sum_{i}\xi_i \mathbf{X}_{i}=0$ and this in turn  yields (after applying the inverse of the Piola transform $\mathcal{J}_{\delta(t)}$) that $\xi=\mathbf{0}$ due to linear independence of the vectors from \eqref{eqn:defn-Xk}. This ensures that $M(t)$ is positive definite.

Now, by eventually denoting $\beta_{n}^{i}:=\left(\mathbf{b}_{n}^{i}\right)^{\prime}$
 it is not difficult to employ a Picard-Lindel\"{o}f result as in Section~\ref{sec:appendix-integro-diff}  since all the terms that appear are locally-Lipschitz in $\beta_n$. This  proves that $\mathbf{b}_n$ exists in an interval $[0,T_0]$ for some $T_0 >0$.

Let us consider the discrete energy
\begin{equation}\label{eqn:def-E(t)}
E_{n}\left(t\right):=\frac{1}{2}\int_{\Omega_{\delta\left(t\right)}}\left|\mathbf{u}_{n}\left(t\right)\right|^{2}dx+\frac{1}{2}\int_{\omega}\left|\partial_{t}\eta_{n}\left(t\right)\right|^{2}+K\left(\eta_{n}\left(t\right)\right),\quad t\in I.
\end{equation}
By multiplying  \eqref{eqn:xk} with $\mathbf{b}_{n}^{k}$ and summing for  $k=\overline{1,n}$ we get that
 \begin{equation}\label{eqn:energy-1}
\frac{d}{dt}E_{n}\left(t\right)+\int_{\Omega_{\delta\left(t\right)}}\left|\nabla\mathbf{u}_{n}\right|^{2}dx=\int_{\Omega_{\delta\left(t\right)}}\mathbf{f}_{n}\cdot\mathbf{u}_{n}dx+\int_{\omega}g_{n}\partial_{t}\eta_{n}dA
\end{equation}
A Gronwall-type argument ensures that 
\begin{equation}\label{eqn:cons-gronwall}
\sup_{t\in\left[0,T\right]}E_{n}\left(t\right)<\infty 
\end{equation}
and this  then allows us to extend $\mathbf{b}_n$ to the whole  interval  $[0,T]$.

Let us also observe that in view of \eqref{eqn:M(t)-pos-def} and \eqref{eqn:cons-gronwall} we can obtain  that
\[
\left\Vert \left(\mathbf{b}_{n}^{k}\right)^{\prime\prime}\right\Vert _{L^{\infty}\left(I\right)}\le C<\infty
\] and thus we obtain 
\[\left\Vert \mathbf{b}_{n}\right\Vert _{C^{2}\left(I\right)}<C\left(n\right)<\infty.
\]

\subsection{The key estimate: uniform energy bound}
The following proposition is the main effort to obtain the existence of a fixed point. Let us denote, for simplicity $\delta=\delta_{n}$, and $\Omega_{\delta}=\Omega_{\delta_{n}}$.

\begin{proposition}\label{prop: invariant-ball} 
Suppose that \[E_{n} (T) \ge E_{n} (0).\]
Then
there exists a constant    
\begin{equation}\label{eqn:R-chap02}
R=c\left(\texttt{data}\right)\cdot\left(C\left(\mathbf{f},g\right)^{2}+C\left(\mathbf{f},g\right)+m^{2}\right)>0
\end{equation}
for which
 \[E_{n} \left(0 \right) \le E_{n} \left(T \right) \le R.\]
By integrating in time from $0$ to $t\in I$ we have
\begin{equation}\label{eqn:chap02-bound-E-n}
\sup_{n\ge1}\sup_{t\in I}E_{n}\left(t\right)\lesssim R+C(\mathbf{f},g)<\infty.
\end{equation}

\end{proposition}
\begin{proof}
The assumption $ E_{n} (0) \le E_{n} (T) $  yields 
\begin{equation}\label{eqn:bad-case}
\int_{I}\int_{\Omega_{\delta\left(t\right)}}\left|\nabla\mathbf{u}_{n}\right|^{2}dxdt\le\int_{I}\int_{\Omega_{\delta\left(t\right)}}\mathbf{f}_{n}\cdot\mathbf{u}_{n}dxdt+\int_{I}\int_{\omega}g_{n}\partial_{t}\eta_{n}dAdt.
\end{equation}
From  \eqref{eqn:bad-case}, Young's inequality,  Poincar\'{e}'s inequality and the trace operator (see Lemma~\ref{lm:trace}) we get
\begin{equation}\label{eqn:consequence-claim}
\begin{aligned}\int_{I}\int_{\Omega_{\delta\left(t\right)}}\left|\mathbf{u}_{n}\right|^{2}dx+\int_{I}\int_{\omega}\left|\partial_{t}\eta_{n}\right|^{2}dAdt & \lesssim\\
\int_{I}\int_{\Omega_{\delta\left(t\right)}}\left|\nabla\mathbf{u}_{n}\right|^{2}dxdt & \lesssim\\
\int_{I}\int_{\Omega_{\delta\left(t\right)}}\left|\mathbf{f}\right|^{2}dxdt+\int_{I}\int_{\omega}\left|g\right|^{2}dAdt.
\end{aligned}
\end{equation}
We can now proceed as in Subsection~\ref{ssec:a-priori-section-paper} (replacing $E(0)=E(T)$ by the weaker assumption $E(T)\ge E(0)$). Testing with $\eta-m\psi$ corresponds to taking the scalar product with $\mathbf{b}_{n}$.
This proves  Proposition~\ref{prop: invariant-ball}.
\end{proof}

\subsection{Fixed-point argument and periodicity}
\subsubsection*{Prescribing an artificially periodic geometry}
Given a function $f\in C^{0}(\overline{I})$ and $\varepsilon>0$,  let us consider the following operator 
\begin{equation}
\mathcal{P}_{\varepsilon}f\left(t\right):=\begin{cases}
f\left(t\right) & 0\le t<T-\varepsilon\\
\gamma_{\varepsilon}\left(t\right) & T-\varepsilon\le t\le 
\end{cases}
\end{equation}
where  $\gamma_{\varepsilon}$ is the line-segment that joins $f(T-\varepsilon)$ and $f(0)$, i.e.
\[
\gamma_{\varepsilon}\left(t\right)=a_{\varepsilon}t+b_{\varepsilon},\quad a_{\varepsilon}=\frac{f\left(0\right)-f\left(T-\varepsilon\right)}{\varepsilon},\ b_{\varepsilon}=f\left(0\right)-\frac{f\left(0\right)-f\left(T-\varepsilon\right)}{\varepsilon}T
\]

\begin{lemma}
The operator $\mathcal{P}_{\varepsilon}$ enjoys the following properties:
\begin{enumerate}[(i)]
\item $\mathcal{P}_{\varepsilon}$ is a continuous operator from $C^{0}(I)$ to itself;
    \item $\left\Vert \mathcal{P}_{\varepsilon}f\right\Vert _{C^{0}\left(I\right)}\le\left\Vert f\right\Vert _{C^{0}\left(\left[0,T-\varepsilon\right]\right)}$;
     \item $\mathcal{P}_{\varepsilon}f$ is time-periodic, that is $\mathcal{P}_{\varepsilon}f\left(0\right)=\mathcal{P}_{\varepsilon}f\left(T\right)$;
    \item  if $f$ is periodic, then $\mathcal{P}_{\varepsilon}f\to f$ as $\varepsilon\to 0$ in $C^{0}(I)$.

\end{enumerate}
\end{lemma}
\begin{proof}
    The first three properties are easy to prove. For (iv) we observe that 
    \begin{equation}
        \begin{aligned}\left|\mathcal{P}_{\varepsilon}f\left(t\right)-f\left(t\right)\right|\le & \sup_{t\in\left[T-\varepsilon,T\right]}\left|a_{\varepsilon}t+b_{\varepsilon}-f\left(t\right)\right|\quad\left(f\left(T\right)=f\left(0\right)\right)\\
\le & \sup_{t\in\left[T-\varepsilon,T\right]}\left|\frac{f\left(T\right)-f\left(T-\varepsilon\right)}{\varepsilon}t+f\left(T\right)-\frac{f\left(T\right)-f\left(T-\varepsilon\right)}{\varepsilon}T-f\left(t\right)\right|\\
\le & \left|\frac{f\left(T\right)-f\left(T-\varepsilon\right)}{\varepsilon}\right|\left(T-t\right)+\sup_{t\in\left[T-\varepsilon,T\right]}\left|f\left(T\right)-f\left(t\right)\right|\\
\le & 2\sup_{t\in\left[T-\varepsilon,T\right]}\left|f\left(T\right)-f\left(t\right)\right|\xrightarrow{\varepsilon\to0}0
\end{aligned}
    \end{equation}
\end{proof}

It is \emph{essential}, in order to obtain the estimates that will be needed below, to prescribe a \emph{time-periodic} domain. Since the given deformation $\delta_n$ from \eqref{eqn:ansatz-u_n-eta_n} is in general not time-periodic, we update it to 
\[
\delta_{n,\varepsilon}=\sum_{k=1}^{n}\mathcal{P}_{\varepsilon}\mathbf{a}_{n}^{k}\left(t\right)X_{k}+m\psi.
\]
In order not to overload the notation, we shall assume that the prescribed geometry is $\Omega_{\delta_{n,\varepsilon}}$ even if denoted $\Omega_{\delta}$.

\subsubsection{Solving the coupling and the periodicity}

Consider the space 
$X:=C^{1}\left(I;\mathbb{R}^{n}\right)$. 
Given $\mathbf{a}_n \in X$, by eventually truncating $\delta$ via  $\min\left\{ \left|\delta\right|,\frac{1}{2}\kappa\right\}$ and 
applying the regularizing operator $\mathcal{R}_{\sigma}$ from Lemma~\ref{lm:regularizers} for $\sigma  >0$,
we may assume that 
\[
\delta=\delta(\varepsilon,\sigma)\in C^{3}\left(I\times\omega\right),\quad\left\Vert \delta\right\Vert _{L_{t,x}^{\infty}}\le\frac{\kappa}{2}.
\]
Then,  let us define the mapping 
\begin{equation}
\mathcal{T}:X\mapsto X,\quad\mathcal{T}:\mathbf{a}_{n}\mapsto\mathbf{b}_{n}
\end{equation}
which assigns to each $\mathbf{a}_{n}$ (and hence to each $\delta_n$) as in \eqref{eqn:ansatz-u_n-eta_n}, the unique solution $\mathbf{b}_n$ of \eqref{eqn:xk} with initial values
\[\mathbf{b}_{n}\left(0\right)=\mathbf{a}_{n}\left(T\right),\quad\mathbf{b}_{n}^{\prime}\left(0\right)=\mathbf{a}_{n}^{\prime}\left(T\right).\]
We then have the following:
\begin{proposition}
The mapping $ \mathcal{T}:X \mapsto X$ admits at least one fixed point.
\end{proposition}
\begin{proof}
We shall prove that $\mathcal{T}$ satisfies the conditions of the  Leray-Schauder Theorem~\ref{thm:Schaeffer}. Note that $\mathcal{T}$ is well defined, recalling the discussion around \eqref{eqn:cons-gronwall}.

The compactness of $\mathcal{T}$ is due to the Arzela-Ascoli theorem and the fact that
\[
\mathcal{T}\left(X\right)\subset C^{2}\left(I\right)\hookrightarrow\hookrightarrow C^{1}\left(I\right).
\] 

That $\mathcal{T}$ is continuous follows from the well-posedness of \eqref{eqn:xk}. 
Indeed, let us consider a sequence $x_{k}\to x$ in $X$ and prove that $\mathcal{T}x_{k}\to\mathcal{T}x$.
The key observation is that if $\mathcal{T}x_{k_{n}}$ is an arbitrary convergent subsequence of $\mathcal{T}x_{k}$, then its limit equals $\mathcal{T}x$. This is because the limit would correspond to a solution of \eqref{eqn:xk} with the initial values and the domain corresponding to $x$. From the well-posedness of the system, this is exactly $\mathcal{T}x$.

To finish the argument, suppose there exists $\varepsilon_0>0$ and a subsequence of $\mathcal{T}x_{k}$ (not relabeled) for which 
\begin{equation}\label{eqn:contrad}
\left\Vert \mathcal{T}x_{k}-\mathcal{T}x\right\Vert _{X}\ge\varepsilon_{0}>0.
\end{equation}
But from \eqref{eqn:cons-gronwall} and \eqref{eqn:xk} it follows that $\mathcal{T}x_{k}$ is bounded in $C^{2}\left(I\right)$. Hence, it has a strongly convergent subsequence. Its limit, however, has to equal $\mathcal{T}x$, as proved before. This contradicts \eqref{eqn:contrad} and proves that $\mathcal{T}$ is continuous.

The most delicate part consists in proving that the set 
\[LS:=\left\{ x\in X:x=\lambda\mathcal{T}x\ \text{for some}\ \lambda\in\left[0,1\right]\right\} 
\] is uniformly bounded. 
To this end, note that in accordance with Proposition~\ref{prop: invariant-ball} it suffices to prove that for any $\mathbf{a}_{n}\in LS$ it holds that $E_{n}\left(T\right)\ge E_{n}\left(0\right)$. 
Since $\mathbf{a}_{n} \in LS$ it follows that 
\begin{equation}\label{eqn:cond-ls}
\left(\mathbf{a}_{n},\mathbf{b}_{n}\left(0\right),\mathbf{b}_{n}^{\prime}\left(0\right)\right)=\lambda\left(\mathbf{b}_{n},\mathbf{b}_{n}\left(T\right),\mathbf{b}_{n}^{\prime}\left(T\right)\right),\quad\text{for}\ \lambda\in\left[0,1\right].
\end{equation}
The case $\lambda=0$ is trivial, so we assume $0<\lambda \le 1$.
Then by  \eqref{eqn:cond-ls} we can compute
\begin{equation}\begin{aligned}E_{n}\left(T\right)-E_{n}\left(0\right)= & \left(\int_{\Omega_{\delta\left(T\right)}}\left|\mathbf{u}_{n}\left(T\right)\right|^{2}dx-\int_{\Omega_{\delta\left(0\right)}}\left|\mathbf{u}_{n}\left(0\right)\right|^{2}dx\right)+\\
 & \int_{\omega}\left(\partial_{t}\eta\left(T\right)\right)^{2}-\left(\partial_{t}\eta\left(0\right)\right)^{2}dA+\\
 & \left(K\left(\frac{1}{\lambda}\eta\left(0\right)\right)-K\left(\eta\left(0\right)\right)\right)\\
\ge & \left(\frac{1}{\lambda^{2}}-1\right)\left[\int_{\Omega_{\delta\left(0\right)}}\left|\mathbf{u}_{n}\left(0\right)\right|^{2}dx+\int_{\omega}\left(\partial_{t}\eta\left(0\right)\right)^{2}dA\right]+\\
 & \left(\frac{1}{\lambda^{2}}-1\right)K\left(\eta\left(0\right)\right)\\
\ge & 0.
\end{aligned}
\end{equation}

\begin{remark}
Note that comparing integrals on $\Omega_{\delta(0)}$ and $\Omega_{\delta(T)}$ is non-trivial when the domains differ. This is precisely why we require $\delta(0)=\delta(T)$, which is ensured by the operator $\mathcal{P}_{\varepsilon}$ above.
\end{remark}

This now places us in the context of Proposition~\ref{prop: invariant-ball}, so the set $LS$ is indeed uniformly bounded. With this, the conditions of the Leray-Schauder fixed-point theorem have been fulfilled and thus there exists a  fixed point of $\mathcal{T}$, that is an $\mathbf{a}_{n} \in X$ with 
\[\mathbf{a}_{n}=\mathbf{b}_{n},\quad\mathbf{b}_{n}\left(0\right)=\mathbf{b}_{n}\left(T\right),\quad\mathbf{b}_{n}^{\prime}\left(0\right)=\mathbf{b}_{n}^{\prime}\left(T\right).\]
\end{proof}
\begin{remark}
Please  observe that Proposition~\ref{prop: invariant-ball} ensures that for  $C\left(\mathbf{f},g\right)$ and $m^{2}$ sufficiently small  we obtain that
\[\left\Vert \delta_{n}\right\Vert _{L_{t,x}^{\infty}}=\left\Vert \eta_{n}\right\Vert _{L_{t,x}^{\infty}}\le \frac{1}{2}\kappa,\]
which means the truncation at level $\kappa/2$ becomes the identity.
\end{remark}
 In this way, we have obtained a time-periodic solution 
 \[
\left(\mathbf{u}_{n},\eta_{n}\right)=\left(\mathbf{u}_{n,\varepsilon,\sigma},\eta_{n,\varepsilon,\sigma}\right)
\] satisfying
\begin{equation}\label{eqn:new-xk}
\begin{aligned}\int_{\Omega_{\eta_{n}\left(t\right)}}\partial_{t}\mathbf{u}_{n}\cdot\mathbf{X}_{k}dx+\frac{1}{2}\int_{\omega}\partial_{t}\eta_{n}\partial_{t}\delta X_{k}dA+\int_{\Omega_{\eta_{n}\left(t\right)}}\nabla\mathbf{u}_{n}\cdot\nabla\mathbf{X}_{k}dx & +\\
\frac{1}{2}\int_{\Omega_{\eta_{n}\left(t\right)}}\left(\mathbf{u}_{n}\cdot\nabla\right)\mathbf{u}_{n}\cdot\mathbf{X}_{k}dx-\frac{1}{2}\int_{\Omega_{\eta_{n}\left(t\right)}}\left(\mathbf{u}_{n}\cdot\nabla\right)\mathbf{X}_{k}\cdot\mathbf{u}_{n}dx & +\\
\int_{\omega}\partial_{tt}\eta_{n}X_{k}dA+\left\langle K^{\prime}\left(\eta_{n}\right),X_{k}\right\rangle  & =\\
\int_{\Omega_{\eta_{n}\left(t\right)}}\mathbf{f}\cdot\mathbf{X}_{k}dx+\int_{\omega}gX_{k}dA\quad\forall t\in I,\quad k=\overline{1,n}.
\end{aligned}
\end{equation}

\subsection{Limiting procedure and proof of Theorem~\ref{thm:main}}\label{ssect:chap02-limit-compactness}
At this point we have obtained discrete solutions for each $n \ge 1$ and $\varepsilon, \sigma>0$, namely $\left(\mathbf{u}_{n,\varepsilon,\sigma},\eta_{n,\varepsilon,\sigma}\right)$.
In order not to overload the notation, we shall only mention the dependence on $n$ for the moment. We first let $n\to\infty$ in \eqref{eqn:new-xk}, and in a second step we consider the limits $\varepsilon, \sigma \to 0$.

\subsubsection{Limit as \texorpdfstring{$n\to\infty$}{n to infinity}}

After the fixed-point procedure we integrate  the relation \eqref{eqn:new-xk} from $0$ to $t$ to get:
\begin{equation}\label{eqn:chap02-discrete-integrated-t}
\begin{aligned}\int_{\Omega_{\eta_{n}\left(t\right)}}\mathbf{u}_{n}\left(t\right)\cdot\mathbf{q}\left(t\right)dx-\int_{\Omega_{\eta_{n}\left(0\right)}}\mathbf{u}_{n}\left(0\right)\cdot\mathbf{q}\left(0\right)dx-\int_{0}^{t}\int_{\Omega_{\eta_{n}\left(s\right)}}\mathbf{u}_{n}\cdot\partial_{t}\mathbf{q}dxds & +\\
\int_{0}^{t}\int_{\Omega_{\eta_{n}\left(s\right)}}\nabla\mathbf{u}_{n}:\nabla\mathbf{q}dxds-\int_{0}^{t}\int_{\Omega_{\eta_{n}\left(s\right)}}\left(\mathbf{u}_{n}\cdot\nabla\right)\mathbf{q}\cdot\mathbf{u}_{n}dxds & +\\
\int_{\omega}\partial_{t}\eta_{n}\left(t\right)b\left(t\right)-\int_{\omega}\partial_{t}\eta_{n}\left(0\right)b\left(0\right)-\int_{0}^{t}\int_{\omega}\partial_{t}\eta_{n}\partial_{t}bdAds & +\\
\int_{0}^{t}\left\langle K^{\prime}\left(\eta_{n}\right),b\right\rangle ds & =\\
\int_{0}^{t}\int_{\Omega_{\eta_{n}\left(s\right)}}\mathbf{f}\cdot\mathbf{q}dxds+\int_{0}^{t}\int_{\omega}gbdAds
\end{aligned}
\end{equation} 
for  all $t\in I$ and all $\left(\mathbf{q},b\right)\in\text{span}\left\{ \left(\varphi\left(t\right)\mathbf{X}_{k},\varphi\left(t\right)X_{k}\right):k=\overline{1,n},\ \varphi\in C_{\text{per}}^{1}\left(I\right)\right\}$.

Setting $t=T$ and using the periodicity of $(\mathbf{u}_n, \eta_n)$ we obtain that
\begin{equation}\label{eqn:chap02-discrete-integrated}
\begin{aligned}-\int_{I}\int_{\Omega_{\eta_{n}\left(s\right)}}\mathbf{u}_{n}\cdot\partial_{t}\mathbf{q}dxds+\int_{I}\int_{\Omega_{\eta_{n}\left(s\right)}}\nabla\mathbf{u}_{n}:\nabla\mathbf{q}dxds & +\\
-\int_{I}\int_{\Omega_{\eta_{n}\left(s\right)}}\left(\mathbf{u}_{n}\cdot\nabla\right)\mathbf{q}\cdot\mathbf{u}_{n}dxds-\int_{I}\int_{\omega}\partial_{t}\eta_{n}\partial_{t}bdAds & +\\
\int_{I}\left\langle K^{\prime}\left(\eta_{n}\right),b\right\rangle ds & =\\
\int_{I}\int_{\Omega_{\eta_{n}\left(s\right)}}\mathbf{f}\cdot\mathbf{q}dxds+\int_{I}\int_{\omega}gbdAds
\end{aligned}
\end{equation}
for all $\left(\mathbf{q},b\right)\in V_{T,\text{per}}^{\eta_{n}}$.

In order to prove Theorem~\ref{thm:main} we need to let $n\to \infty$ in \eqref{eqn:chap02-discrete-integrated}.

Due to the energy estimate~\eqref{eqn:chap02-bound-E-n} and of \eqref{eqn:consequence-claim} there exists a weak limit $(\mathbf{u},\eta)\in V_{S,\text{per}}$ for which we have 
\begin{equation}
    \begin{aligned}\eta_{n}\to\eta & \text{weakly in }L^{\infty}\left(I;H^{2}\left(\omega\right)\right)\cap W^{1,\infty}\left(I;L^{2}\left(\omega\right)\right)\\
\mathbf{u}_{n}\chi_{\Omega_{\eta_{n}}}\to\mathbf{u}\chi_{\Omega_{\eta}} & \text{weakly in }L^{\infty}\left(I;L^{2}\left(\mathbb{R}^{3}\right)\right)\\
\nabla\mathbf{u}_{n}\chi_{\Omega_{\eta_{n}}}\to\nabla\mathbf{u}\chi_{\Omega_{\eta}} & \text{weakly in }L^{2}\left(I;L^{2}\left(\mathbb{R}^{3}\right)\right)
\end{aligned}
\end{equation}

We are left to establish the convergence of nonlinear terms, that is 
\begin{equation}\label{eqn:strong-conv-eta-n}
\left\langle K^{\prime}\left(\eta_{n}\right),\psi\right\rangle \to\left\langle K^{\prime}\left(\eta\right),\psi\right\rangle \quad\forall\psi\in C^{\infty}\left(I\times\omega\right)
\end{equation}
and the \textbf{strong convergence}
\begin{equation}\label{eqn:strong-conv-u-n}
\mathbf{u}_{n}\to\mathbf{u}\quad\text{in}\ L^{2}\left(I;L^{2}\left(\mathbb{R}^{3}\right)\right).
\end{equation}

To prove \eqref{eqn:strong-conv-eta-n} we employ the Aubin-Lions Lemma. Indeed, observe that we have the embeddings
\[
H^{2}\left(\omega\right)\hookrightarrow\hookrightarrow W^{1,3}\left(\omega\right)\hookrightarrow W^{1/2,2}\left(\omega\right)
\]
and therefore the compact embedding
\[
W:=\left\{ \xi\in L^{\infty}\left(I;H^{2}\left(\omega\right)\right),\ \partial_{t}\xi\in L^{2}\left(I;W^{1/2,2}\left(\omega\right)\right)\right\} \hookrightarrow\hookrightarrow L^{\infty}\left(I;W^{1,3}\left(\omega\right)\right).
\]
Due to the energy estimate~\eqref{eqn:chap02-bound-E-n} it follows that the sequence $(\eta_n)_{n\ge 1}$ is bounded in $W$ and therefore has a subsequence (not relabeled) for which 
\begin{equation}
\eta_{n}\to\eta\ \text{in}\ L_{t}^{3}W_{x}^{1,3},
\end{equation} 
which is enough to conclude \eqref{eqn:strong-conv-eta-n}.

On the other hand, the proof of \eqref{eqn:strong-conv-u-n} is much more involved and based on the additional convergence 
\begin{equation}\label{eqn:convergence-L2-norms}
\begin{aligned}\int_{I}\int_{\Omega_{\eta_{n}\left(t\right)}}\left|\mathbf{u}_{n}\right|^{2}dxdt+\int_{I}\int_{\omega}\left|\partial_{t}\eta_{n}\right|^{2}dAdt & \xrightarrow{n\to\infty}\\
\int_{I}\int_{\Omega_{\eta \left(t\right)}}\left|\mathbf{u}\right|^{2}dxdt+\int_{I}\int_{\omega}\left|\partial_{t}\eta\right|^{2}dAdt.
\end{aligned}
\end{equation}

Let us extend $\mathbf{u}_n$ and $\mathbf{u}$ by $\mathbf{0}$ to the whole $\mathbb{R}^{3}$.
Due to \eqref{eqn:chap02-bound-E-n} it follows in particular that 
\begin{equation}
\sup_{n}\left(\left\Vert \mathbf{u}_{n}\right\Vert _{L_{t}^{\infty}L_{x}^{2}\cap L_{t}^{2}H_{x}^{1}}+\left\Vert \partial_{t}\eta_{n}\right\Vert _{L_{t}^{\infty}L_{x}^{2}}\right)\le c<\infty.
\end{equation}
An eventual diagonal argument ensures the existence of a dense, countable subset $I_0 \subset I$ with the property
\begin{equation}
\left(\mathbf{u}_{n}\left(t,\cdot\right),\partial_{t}\eta_{n}\left(t,\cdot\right)\right)\rightharpoonup\left(\mathbf{u}\left(t,\cdot\right),\partial_{t}\eta\left(t,\cdot\right)\right)\ \text{weakly in }L^{2}\left(\mathbb{R}^{3}\right)\times L^{2}\left(\mathbb{\mathbb{R}}^{2}\right)\quad\forall t\in I_{0}.
\end{equation}

For a generic function $b\in H^{2}(\omega)$ with $\int _{\omega} bdA=0$ given by 
\[
b=\sum_{k=1}^{\infty}\left(b,X_{k}\right)_{H^{2}\left(\omega\right)}X_{k},
\]
we truncate it at level $n$ by
\[
b_n=\sum_{k=1}^{n}\left(b,X_{k}\right)_{H^{2}\left(\omega\right)}X_{k}.
\]
Let us consider its truncated extension to $\Omega_{\eta_n}$, denoted by
\[
\mathcal{F}_{\eta_{n}}b=\sum_{k=1}^{n}\left(b,X_{k}\right)_{H^{2}\left(\omega\right)}\mathbf{X}_{k},
\]
and then
\[
\mathcal{F}_{\eta}b=\sum_{k=1}^{\infty}\left(b,X_{k}\right)_{H^{2}\left(\omega\right)}\mathbf{X}_{k}.
\]
Please note that the latter convergence is due to the fact that $\mathcal{F}_{\eta_n}$ and $\mathcal{F}_{\eta}$ are continuous from $W_{x}^{k,p}$ to itself (with a continuity constant depending on $\left\Vert \nabla\eta_{n}\right\Vert _{L_{t,x}^{\infty}}$ and therefore on $\sigma$); see Lemma~\ref{lm:Piola}.

We introduce the following quantities:
\begin{equation}
\begin{aligned}c_{b,n}\left(t\right):= & \int_{\Omega_{\eta_{n}\left(t\right)}}\mathbf{u}_{n}\cdot\mathcal{F}_{\eta_{n}}bdx+\int_{\omega}\partial_{t}\eta_{n}\cdot b_{n}dA\\
c_{b}\left(t\right):= & \int_{\Omega_{\eta\left(t\right)}}\mathbf{u}\cdot\mathcal{F}_{\eta}bdx+\int_{\omega}\partial_{t}\eta\cdot bdA\\
g_{n}\left(t\right):= & \sup_{\left\Vert b\right\Vert _{H^{2}\left(\omega\right)}\le1}\left|c_{b,n}\left(t\right)-c_{b}\left(t\right)\right|.
\end{aligned}
\end{equation}

We claim that
\begin{equation}\label{eqn:chap02-conv-gn-0}
\int_{I}g_{n}\left(t\right)dt\xrightarrow{n\to\infty}0.
\end{equation}

Indeed, we prove that for some $\lambda \in (0, 1)$ we have
\begin{equation}\label{eqn:chap02-claim-Arzela}
\sup_{n\ge1}\left\Vert c_{b,n}\right\Vert _{C^{0,\lambda}\left(I\right)}<\infty.
\end{equation}
First,
\begin{equation}
\left\Vert c_{b,n}\right\Vert _{L_{t}^{\infty}}\lesssim\left(\left\Vert \mathbf{u}_{n}\right\Vert _{L_{t}^{\infty}L_{x}^{2}}+\left\Vert \partial_{t}\eta{}_{n}\right\Vert _{L_{t}^{\infty}L_{x}^{2}}\right)\left\Vert b\right\Vert _{H_{x}^{2}}\lesssim1,
\end{equation}
where we have used \eqref{eqn:chap02-bound-E-n}  and the continuity of $\mathcal{F}_{\eta_n}$.

Then, using \eqref{eqn:chap02-discrete-integrated-t} we obtain that 
\[c_{b,n}\left(t\right)-c_{b,n}\left(s\right)=\int_{s}^{t}\int_{\Omega_{\eta_{n}\left(\tau\right)}}\left(\mathbf{u}_{n}\cdot\nabla\right)\mathcal{F}_{\eta_{n}}b\cdot\mathbf{u}_{n}dxd\tau+\ldots.
\]
We present the estimates only for the most difficult term, namely the convective term, for which we have 
\begin{equation}
\begin{aligned}\left|\int_{s}^{t}\int_{\Omega_{\eta_{n}\left(\tau\right)}}\left(\mathbf{u}_{n}\cdot\nabla\right)\mathcal{F}_{\eta_{n}}b\cdot\mathbf{u}_{n}dxd\tau\right|\le & \left\Vert \mathbf{u}_{n}\right\Vert _{L_{t}^{\infty}L_{x}^{2}}\left\Vert \mathbf{u}_{n}\right\Vert _{L_{t}^{2}L_{x}^{6}}\left\Vert \nabla\mathcal{F}_{\eta_{n}}b\right\Vert _{L_{t}^{\infty}L_{x}^{6}}\left|t-s\right|^{1/2}\\
\lesssim & \left\Vert \mathbf{u}_{n}\right\Vert _{L_{t}^{\infty}L_{x}^{2}\cap L_{t}^{2}H_{x}^{1}}^{2}\left\Vert \mathcal{F}_{\eta_{n}}b\right\Vert _{L_{t}^{\infty}H_{x}^{2}}\left|t-s\right|^{1/2}\\
\lesssim & \left|t-s\right|^{1/2}.
\end{aligned}
\end{equation}
Similar estimates hold for the remaining terms and thus we can use  \eqref{eqn:chap02-claim-Arzela} to infer that, for a fixed $b$ with $\left\Vert b\right\Vert _{H^{2}}\le1$ it holds
\begin{equation}\label{eqn:c-n-strong-conv}
\left\Vert c_{b,n}-c_{b}\right\Vert _{C_{t}^{0}}\xrightarrow{n\to\infty}0.
\end{equation}

The next step consists in proving that 
\begin{equation}
g_{n}\xrightarrow{n\to\infty}0\ \text{in}\ C^{0}\left(I\right).
\end{equation}
To this end, note that \eqref{eqn:chap02-bound-E-n} ensures the following convergence \emph{strongly} in $H^{-1}\left(\mathbb{R}^{3}\right)\times H^{-1}\left(\mathbb{R}^{2}\right)$:
\begin{equation}
\left(\mathbf{u}_{n}\chi_{\Omega_{\eta_{n}}}\left(t,\cdot\right),\partial_{t}\eta_{n}\chi_{\omega}\left(t,\cdot\right)\right)\xrightarrow{n\rightarrow\infty}\left(\mathbf{u}\chi_{\Omega_{\eta}}\left(t,\cdot\right),\partial_{t}\eta\left(t,\cdot\right)\chi_{\omega}\right),\ \forall t\in I_{0}.
\end{equation}

Please notice that uniformly for all $\left\Vert b\right\Vert _{H^{2}}\le1$ it holds that
\begin{equation}
\begin{aligned}\left|c_{b,n}\left(t\right)-c_{b}\left(t\right)\right|\le & \left|\int_{\mathbb{R}^{3}}\left(\mathbf{u}_{n}\chi_{\Omega_{\eta_{n}}}\left(t,\cdot\right)-\mathbf{u}\chi_{\Omega_{\eta}}\left(t,\cdot\right)\right)\mathcal{F}_{\eta_{n\left(t\right)}}bdx\right|+\\
 & \left|\int_{\mathbb{R}^{3}}\mathbf{u}\chi_{\Omega_{\eta}}\left(t,\cdot\right)\left(\mathcal{F}_{\eta_{n\left(t\right)}}b-\mathcal{F}_{\eta\left(t\right)}b\right)dx\right|+\\
 & \left|\int_{\omega}\left(\partial_{t}\eta_{n}-\partial_{t}\eta\right)b_{n}dA\right|+\left|\int_{\omega}\partial_{t}\eta\left(b_{n}-b\right)dA\right|\\
\le & \left\Vert \mathbf{u}_{n}\chi_{\Omega_{\eta_{n}}}\left(t,\cdot\right)-\mathbf{u}\chi_{\Omega_{\eta}}\left(t,\cdot\right)\right\Vert _{H_{x}^{-1}}\left\Vert \mathcal{F}_{\eta_{n\left(t\right)}}b\right\Vert _{H_{x}^{1}}+\\
 & \left\Vert \mathbf{u}\chi_{\Omega_{\eta}}\left(t,\cdot\right)\right\Vert _{H_{x}^{-1}}\left\Vert \mathcal{F}_{\eta_{n\left(t\right)}}b-\mathcal{F}_{\eta\left(t\right)}b\right\Vert _{H_{x}^{1}}+\\
 & \left\Vert \partial_{t}\eta_{n}-\partial_{t}\eta\right\Vert _{H_{x}^{-1}}\left\Vert b_{n}\right\Vert _{H_{x}^{1}}+\left\Vert \partial_{t}\eta\right\Vert _{H_{x}^{-1}}\left\Vert b_{n}-b\right\Vert _{H_{x}^{1}}
\end{aligned}
\end{equation}
and therefore 
\begin{equation}
g_{n}\left(t\right)\xrightarrow{n\to\infty}0\quad\forall t\in I_{0}.
\end{equation}
Finally, for an arbitrary $t\in I$ there is $t_ \varepsilon \in I_0$ with $\left|t-t_{\varepsilon}\right|\le\varepsilon$ so 
\begin{equation}
\begin{aligned}\limsup_{n\to\infty}\left|g_{n}\left(t\right)\right|\le & \limsup_{n\to\infty}\left|g_{n}\left(t\right)-g_{n}\left(t_{\varepsilon}\right)\right|+\limsup_{n\to\infty}\left|g\left(t_{\varepsilon}\right)\right|\\
\lesssim & \varepsilon^{\lambda}.
\end{aligned}
\end{equation}
Since $\varepsilon>0$ was arbitrary this proves that $g_{n}(t)\to 0$ for all $t\in I$.

In order to prove \eqref{eqn:convergence-L2-norms} we establish the two convergences \eqref{eqn:compact1} and \eqref{eqn:compact2} below.

\noindent\textit{Proof of \eqref{eqn:compact1}.}
Since $g_n \to 0$ in $C^0(I)$, the Dominated Convergence Theorem gives
\begin{equation}\label{eqn:gn-conv-zero}
\int_{I}g_{n}\left(t\right)dt\xrightarrow{n\to\infty}0.
\end{equation}
We now choose the test function $b := \left(\partial_{t}\eta_{n}\right)_{\rho} = \partial_{t}\eta_{n}\ast m_{\rho}$, where $m_{\rho}\left(x\right)=\rho^{-2}m\left(\rho^{-1}x\right)$ is a standard mollifier in $\mathbb{R}^{2}$ (functions are extended by zero to the whole space).\footnote{Only mean-value free functions can be extended to divergence-free ones on the fluid domain. For simplicity we neglect the corrector term $\int_{\omega}\left(\partial_{t}\eta_{n}\right)_{\rho}dA$, which vanishes as $\rho\to 0$.}
With this choice, \eqref{eqn:gn-conv-zero} yields that for all $\epsilon>0$ there is $n(\epsilon)>0$ such that for all $n>n(\epsilon)$:
\begin{equation}
\begin{aligned}\int_{I}\left|c_{b,n}\left(t\right)-c_{b}\left(t\right)\right|dt & \le\epsilon\left\Vert b\right\Vert _{H^{2}}\\
 & \le\epsilon\rho^{-4}C\left(m\right)\left\Vert \partial_{t}\eta_{n}\right\Vert _{L_{t}^{\infty}L_{x}^{2}}\\
 & \lesssim\epsilon\rho^{-4}\lesssim\epsilon^{1/2}
\end{aligned}
\end{equation}
for $\rho=\epsilon^{1/8}$. This means
\begin{equation}
\begin{aligned}\int_{I}\int_{\Omega_{\eta_{n}\left(t\right)}}\mathbf{u}_{n}\cdot\mathcal{F}_{\eta_{n}}\left(\partial_{t}\eta_{n}\right)_{\rho}dxdt+\int_{I}\int_{\omega}\partial_{t}\eta_{n}\left(\partial_{t}\eta_{n}\right)_{\rho}dAdt & \xrightarrow{n\to\infty}\\
\int_{I}\int_{\Omega_{\eta\left(t\right)}}\mathbf{u}\cdot\mathcal{F}_{\eta}\left(\partial_{t}\eta\right)_{\rho}dxdt+\int_{I}\int_{\omega}\partial_{t}\eta\left(\partial_{t}\eta\right)_{\rho}dAdt.
\end{aligned}
\end{equation}
Using the mollification convergence $\left(\partial_{t}\eta_{n}\right)_{\rho}\xrightarrow{\rho\to0}\partial_{t}\eta_{n}$ in $L_{t}^{\infty}L_{x}^{2}$ and passing $\rho\to 0$, we obtain
\begin{equation}\label{eqn:compact1}
\begin{aligned}\int_{I}\int_{\Omega_{\eta_{n}\left(t\right)}}\mathbf{u}_{n}\cdot\mathcal{F}_{\eta_{n}}\left(\partial_{t}\eta_{n}\right)dxdt+\int_{I}\int_{\omega}\left(\partial_{t}\eta_{n}\right)^{2}dAdt & \xrightarrow{n\to\infty}\\
\int_{I}\int_{\Omega_{\eta\left(t\right)}}\mathbf{u}\cdot\mathcal{F}_{\eta}\left(\partial_{t}\eta\right)dxdt+\int_{I}\int_{\omega}\left(\partial_{t}\eta\right)^{2}dAdt.
\end{aligned}
\end{equation}

\noindent\textit{Proof of \eqref{eqn:compact2}.}
We need to show that
\begin{equation}\label{eqn:compact2}
\begin{aligned}\int_{I}\int_{\Omega_{\eta_{n}\left(t\right)}}\mathbf{u}_{n}\cdot\left(\mathbf{u}_{n}-\mathcal{F}_{\eta_{n}}\left(\partial_{t}\eta_{n}\right)\right)dxdt & \xrightarrow{n\to\infty}\\
\int_{I}\int_{\Omega_{\eta\left(t\right)}}\mathbf{u}\cdot\left(\mathbf{u}-\mathcal{F}_{\eta}\left(\partial_{t}\eta\right)\right)dxdt.
\end{aligned}
\end{equation}
The key ingredient is the following uniform convergence: for all $\tilde{\sigma}>0$ and all smooth approximations
\[
\delta_{\tilde{\sigma}}\in C_{t,x}^{3}\left(I\times\omega\right),\quad\left\Vert \delta_{\tilde{\sigma}}-\eta\right\Vert _{L_{t,x}^{\infty}}<\tilde{\sigma},\quad\delta_{\tilde{\sigma}}<\eta\ \text{in}\ I\times\omega,
\]
it holds that
\begin{equation}\label{eqn:uniform-conv-h1-piola}
\sup_{\left\Vert \phi\right\Vert _{H_{0,\text{div}}^{1}\left(\Omega\right)}\le1}\left|\int_{\Omega_{\eta_{n}\left(t\right)}}\mathbf{u}_{n}\cdot\mathcal{P}_{n}\mathcal{J}_{\delta_{\tilde{\sigma}}\left(t\right)}\phi\, dx-\int_{\Omega_{\eta\left(t\right)}}\mathbf{u}\cdot\mathcal{J}_{\delta_{\tilde{\sigma}}\left(t\right)}\phi\, dx\right|\xrightarrow{n\to\infty}0
\end{equation}
for all $t\in I$, where the Galerkin projection $\mathcal{P}_n$ acts on $\mathcal{J}_{\delta_{\tilde\sigma}}^{-1}\mathbf{f}:\Omega\mapsto\mathbb{R}^3$ via
\begin{equation*}
\mathcal{P}_{n}\left(\mathcal{J}_{\delta_{\tilde{\sigma}}}^{-1}\mathbf{f}\right):=\sum_{\substack{k\le n \\ k\ \text{even}}}\left(\mathcal{J}_{\delta_{\tilde{\sigma}}}^{-1}\mathbf{f},X_{k}\right)_{H^{1}\left(\Omega\right)}\mathbf{X}_{k},
\end{equation*}
so that for $\mathbf{f}=\mathcal{J}_{\delta_{\tilde{\sigma}}}\phi$ we have \[\mathcal{P}_{n}\mathcal{J}_{\delta_{\tilde{\sigma}}}\phi=\sum_{k\le n,\,k\,\text{even}}\left(\phi,X_{k}\right)_{H^{1}}\mathbf{X}_{k}.
\]
By the continuity of the Piola mapping (Lemma~\ref{lm:Piola}):
\[
\mathcal{P}_{n}\mathbf{f}\in H_{0,\text{div}}^{1}\left(\Omega_{\eta_{n}\left(t\right)}\right),\quad\left\Vert \mathcal{P}_{n}\mathbf{f}\right\Vert _{H^{1}}\lesssim\left\Vert \mathbf{f}\right\Vert _{H^{1}},
\]
with continuity constant depending on $\left\Vert \nabla\eta_{n}\right\Vert _{L^{\infty}}$.

To prove \eqref{eqn:uniform-conv-h1-piola} we use the discrete test function $\left(\mathcal{P}_{n}\mathcal{J}_{\delta_{\tilde\sigma}\left(t\right)}\phi,0\right)$ in \eqref{eqn:chap02-discrete-integrated-t} and estimate the H\"{o}lder modulus of continuity of
\[
t\mapsto\int_{\Omega_{\eta_{n}\left(t\right)}}\mathbf{u}_{n}\cdot\mathcal{P}_{n}\mathcal{J}_{\delta_{\tilde\sigma}\left(t\right)}\phi\, dx.
\]
For the convective term, using $\mathbf{u}_{n}\in L_{t}^{\infty}L_{x}^{2}\cap L_{t}^{2}L_{x}^{6}\hookrightarrow L_{t}^{8/3}L_{x}^{4}$:
\begin{equation}
\left|\int_{s}^{t}\int_{\Omega_{\eta_{n}\left(\tau\right)}}\left(\mathbf{u}_{n}\cdot\nabla\right)\mathcal{P}_{n}\mathcal{J}_{\delta_{\tilde\sigma}\left(\tau\right)}\phi\cdot\mathbf{u}_{n}dxd\tau\right|\lesssim\left\Vert \phi\right\Vert _{W_{x}^{1,2}}\left|t-s\right|^{1/4}.
\end{equation}
For the time-derivative term, writing $\partial_{t}\mathcal{P}_{n}\mathcal{J}_{\eta_{n}}\phi_{n}=T_{1}+T_{2}$ where
\begin{equation}
\begin{aligned}\partial_{t}\mathcal{J}_{\eta_{n}}\phi_{n}= & \frac{\partial_{t}\nabla\psi_{\eta_{n}}\phi_{n}\left(1+\eta_{n}\right)-\nabla\psi_{\eta_{n}}\phi_{n}\partial_{t}\eta_{n}}{\left(1+\eta_{n}\right)^{2}}\circ\psi_{\eta_{n}}^{-1}\\
 & +\nabla\left[\frac{\nabla\psi_{\eta_{n}}\phi_{n}}{1+\eta_{n}}\right]\circ\psi_{\eta_{n}}^{-1}\cdot\partial_{t}\psi_{\eta_{n}}^{-1}=:T_{1}+T_{2},
\end{aligned}
\end{equation}
and using the fact that composition with $\psi_{\eta_{n}}^{-1}$ is continuous from $W^{1,p}$ to itself, we set $\overline{\mathbf{u}}_{n}:=\mathbf{u}_{n}\circ\psi_{\eta_n}$ and estimate via integration by parts (the boundary term vanishes since $\phi\in H_{0,\text{div}}^{1}\left(\Omega\right)$):
\begin{equation}
\begin{aligned}\left|\int_{s}^{t}\int_{\Omega}\overline{\mathbf{u}}_{n}\cdot\partial_{t}\nabla\eta_{n}\phi_{n}dxd\tau\right| & \le\left\Vert \nabla\overline{\mathbf{u}}_{n}\right\Vert _{L_{t}^{2}L_{x}^{2}}\left\Vert \partial_{t}\eta_{n}\right\Vert _{L_{t}^{3}L_{x}^{3}}\left\Vert \phi_{n}\right\Vert _{L_{x}^{6}}\left|t-s\right|^{1/6}\lesssim\left|t-s\right|^{1/6},\\
\left|\int_{s}^{t}\int_{\Omega}\overline{\mathbf{u}}_{n}\cdot\nabla\eta_{n}\phi_{n}\partial_{t}\eta_{n}dxd\tau\right| & \le\left\Vert \overline{\mathbf{u}}_{n}\right\Vert _{L_{t}^{2}L_{x}^{6}}\left\Vert \nabla\eta_{n}\right\Vert _{L_{t}^{\infty}L_{x}^{6}}\left\Vert \phi_{n}\right\Vert _{L_{x}^{6}}\left\Vert \partial_{t}\eta_{n}\right\Vert _{L_{t,x}^{3}}\left|t-s\right|^{1/6}\lesssim\left|t-s\right|^{1/6}.
\end{aligned}
\end{equation}
The estimate for $T_2$ is similar; the term involving $\nabla^{2}\eta_n$ is bounded using $\eta_n\in H_x^3$ (which holds at the regularization level $\sigma>0$):
\begin{equation}
\int_{s}^{t}\int_{\Omega}\left|\overline{\mathbf{u}}_{n}\right|\left|\nabla^{2}\eta_{n}\right|\left|\phi_{n}\right|\left|\partial_{t}\eta_{n}\right|dxd\tau\lesssim_{\sigma}\left|t-s\right|^{1/2}.
\end{equation}
Combining all terms gives a H\"{o}lder modulus of continuity uniform in $n$, and the justification of \eqref{eqn:uniform-conv-h1-piola} then follows by the same argument as for \eqref{eqn:gn-conv-zero}.

Now we use Lemma~\ref{lemma:approx-dual-appendix} with $\psi_{n}:=\mathbf{u}_{n}-\mathcal{F}_{\eta_{n}}\left(\partial_{t}\eta_{n}\right)$: for each $\epsilon>0$ there exist $\sigma_{\epsilon}>0$ and a divergence-free function $\psi_{n,\epsilon}$ satisfying
\[
\psi_{n,\epsilon}\in L^{2}\left(\Omega_{\eta_{n}}\right),\quad\left\Vert \psi_{n,\epsilon}\right\Vert _{L^{2}}\le1,\quad\operatorname{supp}\psi_{n,\epsilon}\subset\Omega_{\eta_{n}-3\sigma_{\epsilon}},\quad\left\Vert \psi_{n}-\psi_{n,\epsilon}\right\Vert _{\left(H^{s}\right)^{\prime}}<\epsilon
\]
for all $s>0$. Its mollification $\left(\psi_{n,\epsilon}\right)_{\rho}$ satisfies $\operatorname{supp}\left(\psi_{n,\epsilon}\right)_{\rho}\subset\Omega_{\eta-\sigma_{\epsilon}}\subset\Omega_{\delta_{\tilde{\sigma}}}$ for sufficiently large $n>n(\epsilon)$ and small $\rho<\rho(\epsilon)$. Hence we may write
\[
\left(\psi_{n,\epsilon}\right)_{\rho}=\mathcal{J}_{\delta_{\tilde{\sigma}}}\phi,\quad\phi\in H_{0,\text{div}}^{1}(\Omega)
\]
for $\tilde{\sigma}<\tilde{\sigma}_{\epsilon}$ small enough. Applying \eqref{eqn:uniform-conv-h1-piola} we obtain that for any $\epsilon>0$ there exists $h_{\epsilon}>0$ such that for all $h\in(0,h_{\epsilon})$ there is $n(h)\ge 1$ with
\begin{equation}\label{eqn:conv-piolatilde-unif}
\begin{aligned}\left|\int_{\Omega_{\eta_{n}\left(t\right)}}\mathbf{u}_{n}\cdot\mathcal{P}_{n}\left(\psi_{n,\epsilon}\right)_{\rho}dx-\int_{\Omega_{\eta\left(t\right)}}\mathbf{u}\cdot\left(\psi_{n,\epsilon}\right)_{\rho}dx\right| & \lesssim h\left\Vert \mathcal{J}_{\delta_{\tilde\sigma}}^{-1}\left(\psi_{n,\epsilon}\right)_{\rho}\right\Vert _{H^{1}}\\
 & \lesssim\frac{h}{\rho^{4}}c\left(\epsilon\right)\left\Vert \left(\psi_{n,\epsilon}\right)_{\rho}\right\Vert _{L^{2}}\lesssim h^{1/2}
\end{aligned}
\end{equation}
for all $n>n(h)$, where we set $\rho=h^{1/8}$ and use that the Piola mapping is an $H^{1}$ isomorphism (Lemma~\ref{lm:Piola}). Thus
\[
\limsup_{n\to\infty}\left|\int_{\Omega_{\eta_{n}\left(t\right)}}\mathbf{u}_{n}\cdot\mathcal{P}_{n}\left(\psi_{n,\epsilon}\right)_{\rho}dx-\int_{\Omega_{\eta\left(t\right)}}\mathbf{u}\cdot\left(\psi_{n,\epsilon}\right)_{\rho}dx\right|=0\quad\forall\epsilon>0,\,\rho\in\left(0,\rho_{\epsilon}\right).
\]
Using the convergences $\left(\psi_{n,\epsilon}\right)_{\rho}\xrightarrow{\rho\to0}\psi_{n,\epsilon}$ in $L_{x}^{2}$ and $\psi_{n,\epsilon}\xrightarrow{\epsilon\to0}\psi_{n}$ in $\left(H_{x}^{s}\right)^{\prime}$, we conclude
\begin{equation}
\int_{I}\int_{\Omega_{\eta_{n}\left(t\right)}}\mathbf{u}_{n}\cdot\left(\mathbf{u}_{n}-\mathcal{F}_{\eta_{n}}\left(\partial_{t}\eta_{n}\right)\right)dxdt\xrightarrow{n\to\infty}\int_{I}\int_{\Omega_{\eta\left(t\right)}}\mathbf{u}\cdot\left(\mathbf{u}-\mathcal{F}_{\eta}\left(\partial_{t}\eta\right)\right)dxdt,
\end{equation}
which is \eqref{eqn:compact2}.

Finally, adding \eqref{eqn:compact1} and \eqref{eqn:compact2} yields
\begin{equation}
\begin{aligned}\int_{I}\int_{\Omega_{\eta_{n}\left(t\right)}}\left|\mathbf{u}_{n}\right|^{2}dxdt+\int_{I}\int_{\omega}\left(\partial_{t}\eta_{n}\right)^{2}dAdt & \xrightarrow{n\to\infty}\\
\int_{I}\int_{\Omega_{\eta\left(t\right)}}\left|\mathbf{u}\right|^{2}dxdt+\int_{I}\int_{\omega}\left(\partial_{t}\eta\right)^{2}dAdt,
\end{aligned}
\end{equation}
which proves the desired compactness \eqref{eqn:convergence-L2-norms}.

Now we can pass to the limit in \eqref{eqn:chap02-discrete-integrated} to obtain that 
\begin{equation}
   \begin{aligned}\int_{I}\int_{\Omega_{\eta\left(s\right)}}-\mathbf{u}\cdot\partial_{t}\mathbf{q}dxds+\nabla\mathbf{u}:\nabla\mathbf{q}-\left(\mathbf{u}\cdot\nabla\right)\mathbf{q}\cdot\mathbf{u}dxds & +\\
\int_{I}\int_{\omega}-\partial_{t}\eta\partial_{t}b+\int_{I}\left\langle K^{\prime}\left(\eta\right),b\right\rangle ds & =\\
\int_{I}\int_{\Omega_{\eta\left(s\right)}}\mathbf{f}\cdot\mathbf{q}dxds+\int_{I}\int_{\omega}gbdAds
\end{aligned}
\end{equation}
for all $\left(\mathbf{q},\xi\right)\in V_{T,\text{per}}^{\eta}$.

\begin{remark}
    It is standard to observe that \eqref{eqn:convergence-L2-norms} yields the \emph{strong} convergence $\left(\mathbf{u}_{n},\eta_{n}\right)\to\left(\mathbf{u},\eta\right)$. This can then easily be  improved   to $\mathbf{u}_{n}\to\mathbf{u}$ in $L^2_{t} L^{4}_{x}$ by interpolating between $L^{2}_{t}L^{2}_{x}$ and $L^{2}_{t}L^{6}_{x}$.
\end{remark}
\subsubsection{Limit as \texorpdfstring{$\varepsilon \to 0$, $\sigma \to 0$}{epsilon, sigma to 0}}\label{subsubsection-compactness2-plate}
Finally, we need to prove that the $L^2$ strong convergence as before is valid also as $\varepsilon\to 0$. To this end, the proof follows the same strategy as above.
Indeed, we can show with analogous arguments as in the case $n\to\infty$ that 
\begin{equation}
\int_{\Omega_{\eta_{\varepsilon}}}\mathbf{u}_{\varepsilon}\cdot\mathcal{F}_{\eta_{\varepsilon}}bdx+\int_{\omega}\partial_{t}\eta_{\varepsilon}bdA\to\int_{\Omega_{\eta}}\mathbf{u}\cdot\mathcal{F}_{\eta}bdx+\int_{\omega}\partial_{t}\eta bdA
\end{equation}
uniformly for all $b\in H^{2}\left(\omega\right)$ and $t\in I$.
The argument is very similar to the one of \eqref{eqn:c-n-strong-conv}.
This allows us to consider in particular $b=\left[\partial_{t}\eta_{\varepsilon}\right]_{\rho}$ with $\rho>0$ a small mollification parameter and then to conclude that 
\begin{equation}\label{eqn:complact-eps-1}
\begin{aligned}\int_{I}\int_{\Omega_{\eta_{\varepsilon}}}\mathbf{u}_{\varepsilon}\cdot\mathcal{F}_{\eta_{\varepsilon}}\left(\partial_{t}\eta_{\varepsilon}\right)dxdt+\int_{I}\int_{\omega}\left(\partial_{t}\eta_{\varepsilon}\right)^{2}dAdt & \to\\
\int_{I}\int_{\Omega_{\eta}}\mathbf{u}\cdot\mathcal{F}_{\eta}\left(\partial_{t}\eta\right)dxdt+\int_{I}\int_{\omega}\left(\partial_{t}\eta\right)^{2}dAdt
\end{aligned}
\end{equation}  as $\varepsilon\to 0$.

The next step is to prove that for any $\tilde{\sigma}>0$ and any smooth $\delta_{\tilde{\sigma}}<\eta$ it holds that 
\begin{equation}
\int_{\Omega_{\eta_{\varepsilon}}}\mathbf{u}_{\varepsilon}\cdot\mathcal{J}_{\delta_{\tilde{\sigma}}}\left(\phi\right)dx\to\int_{\Omega_{\eta}}\mathbf{u}\cdot\mathcal{J}_{\delta_{\tilde{\sigma}}}\left(\phi\right)dx
\end{equation}
uniformly for all $\phi \in H_{0,\text{div}}^{1}$.
Then we may consider $\mathcal{J}_{\delta_{\tilde{\sigma}}}\left(\phi\right)=\psi_{\varepsilon,\theta}$ where $\psi_{\varepsilon,\theta}$ is defined by the approximation Lemma~\ref{lemma:approx-dual-appendix} and estimates 
\[
\left\Vert \mathbf{u}_{\varepsilon}-\mathcal{F}_{\eta_{\varepsilon}}\left(\partial_{t}\eta_{\varepsilon}\right)-\psi_{\varepsilon,\theta}\right\Vert _{\left(H^{s}\left(\mathbb{R}^{3}\right)\right)^{\prime}}<\theta.
\]
Proceeding as in \eqref{eqn:conv-piolatilde-unif}, a similar argument enables us to establish that
\begin{equation}\label{eqn:eps-conv-2}
\int_{I}\int_{\Omega_{\eta_{\varepsilon}}}\mathbf{u}_{\varepsilon}\cdot\left(\mathbf{u}_{\varepsilon}-\mathcal{F}_{\eta_{\varepsilon}}\left[\partial_{t}\eta_{\varepsilon}\right]\right)dxdt\to\int_{I}\int_{\Omega_{\eta}}\mathbf{u}\cdot\left(\mathbf{u}-\mathcal{F}_{\eta}\left[\partial_{t}\eta\right]\right)dxdt.
\end{equation}
From \eqref{eqn:complact-eps-1} and \eqref{eqn:eps-conv-2} by summation 
it follows the $L^{2}$ compactness. Note that from the strong convergence $\mathbf{u}_{\varepsilon}\to\mathbf{u}$ in $L^{2}_{t}L^{2}_{x}$ we can interpolate between $L^{2}_{t}L^{6}_{x}$ to improve the strong convergence to $L^{2}_{t}L^{q}_{x}$ for all $q<6$.

Finally, we only need to repeat the argument leading to \eqref{eqn:convergence-L2-norms} in order to justify the other limit passage  $\sigma \to 0$.

With this, the proof of Theorem~\ref{thm:main} is complete.
\begin{remark}
Please note that we are also able to use the extension operator $\mathcal{F}_{\eta}$ from Proposition~\ref{chap-appendix-prop-extension-plate} and the test function 
\[
\left(\mathcal{F}_{\eta}\left[D_{-h}^{s}D_{h}^{s}\eta-\mathcal{M}_{\eta}\left(D_{-h}^{s}D_{h}^{s}\eta\right)\right],D_{-h}^{s}D_{h}^{s}\eta-\mathcal{M}_{\eta}\left(D_{-h}^{s}D_{h}^{s}\eta\right)\right)
\]
where 
\[D_{h}^{s}\eta=\frac{\eta\left(t,x+he\right)-\eta\left(t,x\right)}{\left|h\right|^{s-1}h},\quad e\in\mathbb{R}^{3},\left|e\right|=1,s>0.\]
This enables us to obtain the additional estimates 
\begin{equation}
\sup_{\varepsilon}\left\Vert \eta_{\varepsilon}\right\Vert _{L_{t}^{2}H_{x}^{2+s}}<\infty
\end{equation}
for all $0<s<\frac{1}{2}$, in the same spirit as \cite{MS22}.
\end{remark}
% \begin{remark}
%     The limit passage $n \to \infty$ is the most tedious part of the proof. The limit passage at a continuous level is much more straightforward; see also \cite[Section 3.1]{LR14} for details of the case of linear membranes, when the problem can be \emph{linearized}, the only nonlinearity being the convective term $\text{div}\left(\mathbf{u}\otimes\mathbf{u}\right)$ which is canonically linearized by $\text{div}\left(\mathbf{u}\otimes\mathbf{v}\right)$. In this way, one avoids the limit passage at the discrete level, weak convergence being sufficient. The Navier-Stokes system is then recovered through a fixed-point $\mathbf{v}\mapsto \mathbf{u}$. When dealing with nonlinear membranes the arguments have to be refined and the limit passage at the discrete level seems unavoidable; see also \cite{MS22}.
% \end{remark}

%\newpage
\section{Auxillary tools}\label{chap:appendix}
In this section we gather various auxiliary results used throughout the proofs above.
\subsection{Reynolds' transport theorem}
We  recall the classical  Reynolds' Transport Theorem, which reads as follows
\begin{theorem}\label{thm:Reynolds-chap-appendix}
For all $g=g(t,x)$ such that the following quantities are well defined, it holds that 
\begin{equation}
    \frac{d}{dt}\int_{\Omega_{\eta\left(t\right)}}gdx=\int_{\Omega_{\eta\left(t\right)}}\partial_{t}gdx+\int_{\partial\Omega_{\eta\left(t\right)}}g\mathbf{v}\cdot\nu_{\eta}dA.
\end{equation}
\end{theorem}
Here $\mathbf{v}$ is the speed of the boundary $\partial\Omega_{\eta(t)}$, so $\mathbf{v}\left(t,\cdot\right)=\left(\partial_{t}\eta\nu\right)\circ\phi_{\eta\left(t\right)}^{-1}$.
\begin{proof}
    See e.g. \cite{Gurtin81}.
\end{proof} 
% \subsection{Poincar\'{e}'s inequality}
% \begin{theorem}\cite[Theorem 13.15]{Le17} \label{thm:poincare} If $\Omega \subset \mathbb{R}^{N}$ is an open set that lies between two hyperplanes which are at distance $d$, $m \in \mathbb{N}$ and $1 \le p < \infty$, then there is a constant $c=c(m,N,p)>0$ such that \[\int_{\Omega}\left|\nabla^{k}\mathbf{u}\right|^{p}dx\le c\frac{d^{p\left(m-k\right)}}{p\left(m-k\right)}\int_{\Omega}\left|\nabla^{m}\mathbf{u}\right|^{p}dx\quad\forall \mathbf{u}\in W_{0}^{m,p}(\Omega),0\le k\le m-1.\] 
% \end{theorem} 
\subsection{Korn's identity}
\begin{lemma}\label{lm:Korn} If we denote $D\left(\mathbf{u}\right):=\frac{1}{2}\left(\nabla\mathbf{u}+\left(\nabla\mathbf{u}\right)^{T}\right)$ then it holds that \[\left\Vert \nabla\mathbf{u}\right\Vert _{L^{2}\left(\Omega_{\eta\left(t\right)}\right)}^{2}=2\left\Vert D\left(\mathbf{u}\right)\right\Vert _{L^{2}\left(\Omega_{\eta\left(t\right)}\right)}^{2} \]
and more general, 
\[
2\int_{\Omega_{\eta}}D\left(\mathbf{u}\right):D\left(\mathbf{q}\right)dx=\int_{\Omega_{\eta}}\nabla\mathbf{u}:\nabla\mathbf{q}dx
\]
for all $\mathbf{q}\in W^{1,2}\left(\Omega_{\eta}\right)$ with $\text{tr}_{\eta}\left(\mathbf{q}\right)=\xi\nu$ for some scalar function $\xi$.
\end{lemma}
\begin{proof}
    See e. g. \cite[Lemma A.5]{LR14}.
\end{proof}
\subsection{An existence result for integro-differential equations}\label{sec:appendix-integro-diff}
\begin{lemma}
Consider the system 
\begin{equation}\label{eqn:chap-appendix-integro-diff}
\mathbf{z}^{\prime}\left(t\right)=\mathbf{B}\left(t,\mathbf{z}\left(t\right)\right)+\int_{I}\mathbf{C}\left(t,s,\mathbf{z}\left(s\right)\right)ds+\mathbf{D}\left(t\right),\quad\mathbf{z}\left(0\right)=\mathbf{z}_{0}
\end{equation}
with the unknown $\mathbf{z}\in C^{1}\left(\mathbb{R_{+}};\mathbb{R}^{n}\right)$ and $\mathbf{z}_{0} \in \mathbb{R}^{n}$. Suppose moreover that 
\[\mathbf{B}\in C\left(\mathbb{R}_{+}\times\mathbb{R}^{n};\mathbb{R}^{n}\right),\quad\mathbf{B}\in C\left(\mathbb{R}_{+}\times\mathbb{R}_{+}\times\mathbb{R}^{n};\mathbb{R}^{n}\right),\quad\mathbf{D}\in C\left(\mathbb{R}_{+};\mathbb{R}^{n}\right)\]
with $\mathbf{B,C}$ locally-Lipschitz in the last variable. Then there exists a time $T_0 >0$ for which the system \eqref{eqn:chap-appendix-integro-diff} admits a solution. The solution may be prolonged to a maximal existence interval $\left[0,T^{*}\right]$ provided that $\lim\limits_{t\to T^{*}}\left|\mathbf{z}\left(t\right)\right|=\infty$.
\end{lemma}
\begin{proof}
  It foloows by the Picard-Lindel\"{o}f principle, considering the mapping 
  \[
\mathbf{z}\left(\tau\right)\mapsto\mathbf{z}_{0}+\int_{0}^{\tau}\mathbf{B}\left(t,\mathbf{z}\left(t\right)\right)+\int_{I}\mathbf{C}\left(t,s,\mathbf{z}\left(s\right)\right)ds+\mathbf{D}\left(t\right)dt
  \]
  and proving that it is a contraction on a small interval $[0,T_0]$.
\end{proof}

\subsection{The Leray-Schauder  fixed point theorem}\label{thm:leray-schauder}
\begin{theorem}\label{thm:Schaeffer}
Let $X$ be a Banach space and   $A :  X \mapsto X$ be a continuous and compact mapping such that the set \[\left\{ x\in X:x=\lambda Ax\ \text{for some}\ \lambda\in\left[0,1\right]\right\} \]
is bounded. Then $A$ has at least one fixed point.
\end{theorem}
\begin{proof}
    See e.g. \cite[Theorem 4]{Ev10}.
\end{proof}
%\subsection{The extension operator}
%For  a continuous mapping $\delta : \omega \mapsto \mathbb{R}$  we denote  \begin{equation}\label{eqn:def-tau(eta)}
%\tau\left(\delta\right):=\begin{cases}
%\frac{\kappa}{\kappa-\left\Vert \delta\right\Vert _{L^{\infty}\left(\omega\right)}} & \left\Vert \delta\right\Vert _{L^{\infty}\left(\omega\right)}<\kappa\\
%\infty & \text{else}
%\end{cases}. 
%  \end{equation}
%Following \cite[Lemma 2.6]{LR14} an extension operator  from $\Omega$ to $\Omega_{\eta}$ and the corresponding Lebesgue and Sobolev spaces is constructed where $\eta \in H^{2}(\omega)$. Since this is a H\"{o}lder continuous function but not a Lipschitz continuous one, there is a small loss in the order of integrability. %However when $\eta \in C^{\infty}(\omega)$ we can overwhelm this loss.
%We have
%\begin{lemma}\label{lm:extension}
%For any $1<p \le \infty$ and $\eta \in W^{1,\infty}(\omega)$ with $\left\Vert \eta\right\Vert _{L^{\infty}\left(\omega\right)}<\kappa$ the linear mapping $\text{Ext} : v \mapsto v \circ \psi_{\eta}$ is continuous from $L^{p}(\Omega_{\eta})$ to $L^{p} (\Omega)$ and from $W^{1,p}(\Omega_{\eta})$ to $W^{1,p}(\Omega)$. Analogously when $\psi_{\eta}$ is replaced by $\psi_{\eta} ^{-1}$. The continuity constants depend on $\Omega$, $p$ and a bound for $\left\Vert \eta\right\Vert _{H^{2}\left(\omega\right)}$ and $\tau (\eta)$.
%\begin{proof}
%It follows from the mentioned reference, keeping in mind that $\eta$ is Lipschitz continuous in this case.
%\end{proof}
%\end{lemma}
\subsection{The Aubin-Lions lemma}
\begin{lemma}
Let $X_0,\  X,\  X_1$ be three Banach spaces with $X_1$ reflexive and  such that $X_{0}\hookrightarrow\hookrightarrow X\hookrightarrow X_{1}$. Then we have the compact embedding
\[
W:=\left\{ u\in L^{p}\left(I;X_{0}\right),\ u^{\prime}\in L^{q}\left(I;X_{1}\right)\right\} \hookrightarrow\hookrightarrow L^{p}\left(I;X\right)
\]
for $p,q>1$.\end{lemma}
\begin{proof}
    See e.g. \cite{Seregin}.
\end{proof}
\subsection{The Piola mapping}
We follow \cite[Remark 2.5]{LR14} and state the following
\begin{lemma}\label{lm:Piola}
 Let $\delta \in C^{2}(\omega)$  with $\left\Vert \delta\right\Vert _{L^{\infty}\left(\omega\right)}<\kappa$ and $\mathbf{g}: \Omega \mapsto\mathbb{R}^{3}$ . 
 Let $\psi_{\delta}:\Omega\mapsto\Omega_{\delta}$ be a diffeomorphism from the refernce configuration $\Omega$ to the deformed one $\Omega_{\delta}$.
 
 We define the    \emph{Piola transform} of $\mathbf{g}$ under $\psi_\delta$ by 
 \[
\mathcal{J}_{\delta}\mathbf{g}:=\left(\nabla\psi_{\delta}\left(\det\nabla\psi_{\delta}\right)^{-1}\mathbf{g}\right)\circ\psi_{\delta}^{-1}.
 \]
Then, for any $k\in\left\{ 0,1\right\} ,\ p\ge1$ the operator 
\[
\mathcal{J}_{\delta}:W^{k,p}\left(\Omega\right)\mapsto W^{k,p}\left(\Omega_{\delta}\right),\quad\mathbf{g}\mapsto\mathcal{J}_{\delta}\mathbf{g}
\]
defines an isomorphism which  moreover preserves the space periodic boundary values and the divergence-free condition. 

% Furthermore, in the case of elastic plates if $\mathbf{g}\in L^{p}\left(I;L^{q}\left(\Omega\right)\right)$ then 
% \[
% \left\Vert \partial_{t}\mathcal{J}_{\delta}\mathbf{g}\right\Vert _{L_{t}^{p}L_{x}^{q}}\le\left\Vert \partial_{t}\mathbf{g}\right\Vert _{L_{t}^{p}L_{x}^{q}}+\left\Vert \mathbf{g}\partial_{t}\delta\right\Vert _{L_{t}^{p}L_{x}^{q}}+\left\Vert \left|\nabla\mathbf{g}\right|\partial_{t}\delta\right\Vert _{L_{t}^{p}L_{x}^{q}}+\left\Vert \mathbf{g}\left|\nabla\delta\right|\partial_{t}\delta\right\Vert _{L_{t}^{p}L_{x}^{q}}
% \]
\end{lemma}
\begin{proof}
Follows by elementary computations.
\end{proof}

% \begin{proof} It can be shown that 
% $\partial_{y}\mathcal{J}_{\delta}\varphi\circ\psi_{\delta}=\left(\det d\psi_{\delta}\right)^{-1}\partial_{i}\varphi\ $ and thus $\text{div}\mathcal{J}_{\delta}\varphi\circ\psi_{\delta}=\left(\det d\psi_{\delta}\right)^{-1}\text{div}\varphi$.
% \end{proof}

\subsection{The trace operator}
In order to  rigorously justify the traces  of  functions  defined on $\Omega_{\eta(t)}$ we use 
\begin{lemma}\label{lm:trace}  If $1 < p \le \infty$ and $\eta \in H^{2}(\omega)$ with $\left\Vert \eta\right\Vert _{L^{\infty}\left(\omega\right)}<\kappa$, then for any $r \in (1,p)$ the mapping \[\text{tr}_{\eta}:W^{1,p}\left(\Omega_{\eta}\right)\mapsto W^{1-\frac{1}{r},r}\left(\omega\right),\quad\text{tr}_{\eta}\left(v\right):=\left(v\circ\psi_{\eta}\right)_{|\omega}\] is well defined and continuous, with continuity constant depending on $\Omega, r, p$ and a bound for $\left\Vert \eta\right\Vert _{H^{2}\left(\omega\right)}$ and $\kappa$.
\end{lemma}
  \begin{proof}
    See e.g.\cite[Corollary 2.9]{LR14}
  \end{proof}
\subsection{The Bogovskii operator}
\begin{theorem}\label{thm-chap04:Bogovskii} Let $\Omega \subset \mathbb{R}^{N}$ be a bounded Lipschitz domain. Then for all $m \ge 0$ and all $1<q<\infty$ there exists a linear and continuous operator 
\[B:\left\{ f\in W_{0}^{m,q}\left(\Omega\right):\int_{\Omega}f=0\right\} \mapsto\left[W_{0}^{m+1,q}\left(\Omega\right)\right]^{N}\quad\text{and}\ \begin{cases}
\text{div}Bf=f & \text{in}\ \Omega\\
Bf=0 & \text{on}\ \partial\Omega
\end{cases}\]
\end{theorem}
\begin{proof}
    See e. g.  \cite{Bog} or \cite[Section 3.3]{Galdi-book}.
\end{proof}  
\subsection{Divergence-free extension operators}

Following \cite{Gr05, MS22} we extend each function $\xi:W^{1,1}\left(\omega\right)\mapsto\mathbb{R},\quad\int_{\omega}\xi dA=0$ to a divergence-free vector-valued one, with zero boundary values on the bottom $\Gamma_b = \omega\times\{0\}$. In the periodic setting $\omega=\mathbb{T}^{2}$ no boundary condition on $\partial\omega$ is needed. More precisely for a sufficiently large $M>0$ and $\Omega_{M}:=\omega\times\left(0,M\right)$ such that $\Omega_{\eta\left(t\right)}\subset \Omega_M$ for all $t\in I$
it holds the following result.
\begin{proposition}\label{chap-appendix-prop-extension-plate}
Let 
\[
\eta\in L^{\infty}\left(I;L_{x}^{\infty}\right),\quad\left\Vert \eta\right\Vert _{L_{t,x}^{\infty}}\le\kappa<1.
\]
There exists an operator 
\begin{equation}
\mathcal{F}_{\eta}:\left\{ \xi\in W^{1,1}\left(\omega\right),\int_{\omega}\xi dA=0\right\} \mapsto W_{\text{div}}^{1,1}\left(\Omega_{M=2}\right)
\end{equation}
such that $\text{div}_{x}\mathcal{F}_{\eta}\xi=0$ and $\left(\mathcal{F}_{\eta}\xi,\xi\right)\in V_{T}^{\eta}$ and 
with the property
\[
\left\Vert \mathcal{F}_{\eta}\xi\right\Vert _{W_{t}^{1,p}W_{x}^{k,q}}\lesssim\left\Vert \xi\right\Vert _{W_{t}^{1,p}W_{x}^{k,q}}.
\]
The operator can be extended to $L^2(\omega)$ by approximation with smooth functions.
A general function $\xi\in W^{1,1}\left(\omega\right)$ can be made  mean-value free via the operator
\begin{equation}
    \mathcal{M}_{\eta}\xi:=\xi-\psi\int_{\omega}\xi dA,\quad\psi\in C^{\infty}\left(\omega;\mathbb{R}\right),\quad\int_{\omega}\psi dA=1.
\end{equation}
\end{proposition}
\begin{proof}
We consider a smooth function 
    \[
\sigma\in C^{\infty}\left(\mathbb{R};\mathbb{R}_{+}\right),\quad\sigma\left(z\right)=\begin{cases}
1 & z\in\left(1-\kappa,1+\kappa\right)\\
0 & z\in\left(0,\frac{1-\kappa}{2}\right)
\end{cases}
\]
and define 
\begin{equation}
\overline{\mathcal{F}_{\eta}}\left(\xi\right)\left(x,y,z\right):=\xi\sigma\left(z\right)\mathbf{e}_{3}=\left(0,0,\xi\left(x,y\right)\sigma\left(z\right)\right).
\end{equation}

We see that 
\[
\text{div}\overline{\mathcal{F}_{\eta}}\xi\left(x,y,z\right)=\sigma^{\prime}\left(z\right)\xi\left(x,y\right)\chi_{\left\{ \frac{1-\kappa}{2}<z<1-\kappa\right\}}.
\] 

We need to correct  $\text{div}\overline{\mathcal{F}_{\eta}}\xi\left(x,y,z\right)$  when $\frac{1-\kappa}{2}<z<1-\kappa$. 
For this, we consider the Bogovskii operator $\text{Bog}$ from Theorem\ref{thm-chap04:Bogovskii}
and we consider

\[
\mathcal{F}_{\eta}\xi:=\overline{\mathcal{F}_{\eta}}\xi-\text{Bog}\ \text{div}\overline{\mathcal{F}_{\eta}}\xi.
\]
 This is compatible with the definition of the Bogovskii operator since 
 \[
 \int_{\omega\times\left(\frac{1-\kappa}{2}<z<1-\kappa\right)}\text{div}\overline{\mathcal{F}_{\eta}}\xi dx=\int_{\omega}\xi\left(x,y\right)dA\cdot\left.\sigma\left(z\right)\right|_{z=\frac{1-\kappa}{2}}^{z=1-\kappa}=0.
 \]
\end{proof}

\subsection{A regularizing operator}\label{ssec:regularizing-op}
\begin{lemma}
\label{lm:regularizers}
There exists an operator \[\mathcal{R}_{\varepsilon}:C\left(I\times\omega\right)\mapsto C^{4}\left(I\times\omega\right)\] such that, for all  $\delta\in C\left(I\times\omega\right)$ it holds:
\begin{enumerate}[(a)]
\item $\mathcal{R}_{\varepsilon}\delta\to\delta$ uniformly as $\varepsilon \to 0$;
\item $\mathcal{R}_{\varepsilon}:L^{2}\left(0,T;H^{2}\left(\omega\right)\right)\mapsto L^{2}\left(0,T;H^{2}\left(\omega\right)\right)$ and $\mathcal{R}_{\varepsilon}\delta\to\delta$ in $L^{2}\left(0,T;H^{2}\left(\omega\right)\right)$ as $\varepsilon \to 0$;
\item If $\partial_{t}\delta\in L^{p}\left(\left(0,T\right)\times\omega\right)$ then $\partial_{t}\mathcal{R}_{\varepsilon}\delta=\mathcal{R}_{\varepsilon}\partial_{t}\delta\to\partial_{t}\delta$ in $L^{p}\left(\left(0,T\right)\times\omega\right)$ as $\varepsilon \to 0$;
\item If $\delta\in C^{\gamma}\left(\left(0,T\right)\times\omega\right)$ for some $\gamma \in (0,1)$ then $\mathcal{R}_{\varepsilon}\delta\to\delta$ in $C^{\gamma}\left(\left(0,T\right)\times\omega\right)$ as $\varepsilon \to 0$;
\item $\left\Vert \mathcal{R}_{\varepsilon}\delta\right\Vert _{L^{\infty}\left(\left(0,T\right)\times\omega\right)}\le\left\Vert \delta\right\Vert _{L^{\infty}\left(\left(0,T\right)\times\omega\right)}$.
\end{enumerate}
\end{lemma}
\begin{proof}
   See e.g. \cite{LR14},\cite{BS18}.
\end{proof}
% We will also regularize functions $\mathbf{v} \in L^{2} ((0,T) \times \mathbb{R}^{3})$ by convolution with standard mollifying kernels. Thus, if $\psi_\varepsilon$ denote the standard time-space mollification kernels, we set \[\mathcal{R}_{\varepsilon}\mathbf{v}\left(t,x\right):=\int_{\mathbb{R}^{3+1}}\psi_{k}\left(t-s,x-y\right)\chi_{\left(0,T\right)\times\Omega_{\mathcal{R}_{\varepsilon}\delta}}\left(s,y\right)\mathbf{v}\left(s,y\right)dsdy.\]

\subsection{ Sobolev embedding on moving domains}
\begin{lemma}
\label{lm:Sobolev-embedding}
If $1 < p<3$ and $\eta \in H^{2}(\omega)$ with $\left\Vert \eta\right\Vert _{L^{\infty}\left(\omega\right)}<\kappa$ then \[W^{1,p}\left(\Omega_{\eta}\right)\hookrightarrow L^{p^{*}}\left(\Omega_{\eta}\right)\ \text{with} \   p^{*}:=\frac{3p}{3-p}\] and the embedding constant depends on $\Omega$, $p$ and a bound for $\left\Vert \eta\right\Vert _{H^{2}\left(\omega\right)}$.
%and $\left\Vert \eta\right\Vert _{L^{\infty}\left(\omega\right)}$
\end{lemma}
\begin{proof}
    See e.g. \cite[Corollary 2.10]{LR14}.
\end{proof}
% \subsection{The Kakutani-Glicksberg-Fan  fixed point theorem}
% \begin{theorem}[Kakutani-Glicksberg-Fan]\label{thm: Kakutain-chap-appendix} 
% Let $C$ be a convex subset of a normed vector space $Z$ and let $F:C \to \mathcal{P}(C)$ be a  set-valued mapping which has closed graph. Moreover, let $F(C)$ be contained in a compact subset of $C$, and let $F(z)$ be non-empty, convex, and compact for all $z \in C$. Then $F$ possesses a fixed point, that is there is $c_0 \in C$ with $c_0 \in F(c_0)$. 
% \end{theorem}
% We say a set-valued mapping $F : C \mapsto \mathcal{P}(C)$ has \emph{closed graph} provided that the set $\left\{ \left(x,y\right):y\in F\left(x\right)\right\} $ is closed in $X \times Y$ with the product topology or that, equivalently, for any sequences $x_n \to x$ and $y_n \to y$ with $y_n \in F(x_n)$ for any $ \ge 1$ it follows that $y \in F(x)$.
% \begin{proof}
%     See e. g.  \cite[Chapter 2, Section 5.8]{GD03}.
% \end{proof}
\subsection{An approximation lemma}
  Let  
$H\left(\Omega_{\eta}\right):=\left\{ \mathbf{f}\in C^{\infty}\left(\Omega_{\eta}\right),\ \text{div }\mathbf{f}=0,\quad\mathbf{f}=\mathbf{0}\ \text{on}\ \Gamma_{b}\right\} ^{\left\Vert \cdot\right\Vert _{2}}$ for $\eta \in H^{2}(\omega)$. It holds the following
\begin{lemma}\label{lemma:approx-dual-appendix}
Let $s>0$ and $\varepsilon>0$. Then there is $\sigma_{\varepsilon}>0$ such that for all $\phi \in H(\Omega_{\eta})$ with $\left\Vert \phi\right\Vert _{L^{2}\left(\Omega_{\eta}\right)}\le1$ there is $\psi\in H\left(\Omega_{\eta}\right)$ with 
\[\text{supp} \ \psi\subset\Omega_{\eta-\sigma_{\varepsilon}},\quad\left\Vert \psi\right\Vert _{L^{2}\left(\Omega_{\eta}\right)}\le2,\quad\left\Vert \phi-\psi\right\Vert _{\left(H^{s}\left(\mathbb{R}^{3}\right)\right)^{\prime}}<\varepsilon.\]
\end{lemma}
\begin{proof}
    It follows by a contradiction argument which is standard, making use of the compact embedding $L^{2}\hookrightarrow\hookrightarrow\left(H^{s}\right)^{\prime}$ for $s>0$. A proof of a more general result can be found in \cite[Lemma A13]{LR14}.
\end{proof}
\subsection{The Stokes operator}
Consider the Stokes problem 
\begin{equation}\label{eqn:chap05-Stokes}
\begin{cases}
-\Delta\mathbf{u}+\nabla\pi=\mathbf{F} & \text{in}\ \Omega\\
\text{div}\mathbf{u}=0 & \text{in}\ \Omega\\
\mathbf{u}=\mathbf{g} & \text{on}\ \partial\Omega.
\end{cases}
\end{equation}
The following existence and regularity result can be found in \cite{Galdi-book}. In our setting, this theorem is applied to the reference domain $\Omega$ at the Galerkin approximation level, where the lateral boundary conditions are periodic.
\begin{theorem}\label{chap05-thm-Stokes-reg}
Let $\Omega \subset \mathbb{R}^{N}$ be a bounded domain with boundary of class $C^{k+2}$,  $k \in \mathbb{N}$. Let $1<p<\infty$, $\mathbf{F}\in W^{k,p}(\Omega)$ and $\mathbf{g}\in W^{k+2-\frac{1}{p},p}(\partial\Omega)$. Then there exists a unique solution $\left(\mathbf{u},\pi\right)\in W^{k+2,p}\left(\Omega\right)\times W^{k+1,p}\left(\Omega\right)$ such that 
\[
\left\Vert \mathbf{u}\right\Vert _{W^{k+2,p}\left(\Omega\right)}+\left\Vert \pi\right\Vert _{W^{k+2,p}\left(\Omega\right)}\lesssim_{k,N,p,\Omega}\left\Vert \mathbf{F}\right\Vert _{W^{k,p}\left(\Omega\right)}+\left\Vert \mathbf{g}\right\Vert _{W^{k+2-\frac{1}{p},p}\left(\partial\Omega\right)}.
\]

\end{theorem}
\subsection{Eigenvectors of a compact, symmetric operator}
\begin{theorem}\label{thm:spectrum-compact-symmetric}
Let $H$ be a separable Hilbert space and $S:H\mapsto H$ be a compact and symmetric operator.Then there exists a countable orthonormal basis of $H$, formed by eigenvectors of $S$.
\end{theorem}
\begin{proof}
  See \cite[Appendix E, Theorem 7]{Ev10}  
\end{proof}
\paragraph{Acknowledgements.}
 
The author used the AI assistant Claude (Anthropic) to help edit and polish
the written text of this manuscript for spelling, grammar, and general style.
All mathematical content, results, and proofs are entirely the author's own work.

\bibliographystyle{plain}
\bibliography{bibliography} 

\end{document}